\documentclass[10pt,twoside]{article}

\usepackage[left=4.0cm,right=3.5cm,top=3cm,bottom=2.5cm]{geometry}
\usepackage{amsfonts,amssymb}
\usepackage{amsmath,amsthm}
\usepackage{graphicx}
\usepackage{cite}
\usepackage{bm}
\usepackage{enumerate}
\usepackage[utf8]{inputenc}
\usepackage[english]{babel}
\usepackage{authblk}
 \usepackage[backref,colorlinks,linkcolor=black,anchorcolor=green,citecolor=black]{hyperref}
\usepackage{autonum}
\usepackage{float}
\usepackage{cases}
\usepackage{mathrsfs}
\usepackage{upgreek}
\usepackage{subcaption}

\usepackage{bm}

\numberwithin{equation}{section}

  \newtheorem{theorem}{Theorem}
  \newtheorem{lemma}[theorem]{Lemma}
  \newtheorem{corollary}[theorem]{Corollary}
  \newtheorem{proposition}[theorem]{Proposition}
  \newtheorem{definition}[theorem]{Definition}
  \newtheorem{remark}[theorem]{Remark}

\newcommand{\der}[2][]{\tfrac{\mathrm{d}#1}{\mathrm{d}#2}}

\numberwithin{theorem}{section}

\usepackage{fancyhdr}

   \allowdisplaybreaks
    
\renewcommand{\vec}[1]{\boldsymbol{\mathrm{#1}}}              
     
\begin{document}

\title{\bf{\normalsize{WELL-POSEDNESS AND EQUILIBRIUM BEHAVIOUR OF OVERDAMPED DYNAMIC DENSITY FUNCTIONAL THEORY}}}

          \author{{\small{B. D. GODDARD}\thanks{School of Mathematics and Maxwell Institute for Mathematical Sciences, University of
Edinburgh, Edinburgh EH9 3FD, UK. Corresponding author: R. D. Mills-Williams (r.mills@ed.ac.uk).}},
          \and {\small{R. D. MILLS-WILLIAMS}\thanks{Institut f{\"u}r Theoretische Physik II: Weiche Materie, Heinrich--Heine--Universit{\"a}t D{\"u}sseldorf,  Universit{\"a}tsstra{\ss}e 1, D--40225
D{\"u}sseldorf, Germany.}}, \and {\small{M. OTTOBRE}\thanks{Heriot Watt University, Mathematics Department,
Edinburgh, EH14 4AS, UK.}}, \and {\small{G. A. PAVLIOTIS}}\thanks{Department of Mathematics, Imperial College London, London SW7 2AZ, UK. } 
}

\date{}

\maketitle

\begin{abstract}
We establish the global well-posedness of overdamped dynamic density functional theory (DDFT): a nonlinear, nonlocal integro-partial differential equation used in statistical mechanical models of colloidal fluids, and other applications including nonlinear reaction-diffusion systems and opinion dynamics. With nonlinear no-flux boundary conditions, we determine the existence and uniqueness of the weak density and flux, subject to two-body hydrodynamic interactions (HI). We also show that the density is Lyapunov stable with respect to the usual (Helmholtz) free energy functional. Principally, this is done by rewriting the dynamics for the density in an implicit gradient flow form, resembling the classical Smoluchowski equation but with spatially inhomogeneous diffusion and advection tensors. We also rigorously show that the stationary density is independent of the HI tensors, and prove exponentially fast convergence to equilibrium.
\end{abstract}




\maketitle

\section{Introduction}\label{intro}

For equilibrium fluids, there is a rigorous mathematical framework proving the existence of nontrivial fluid densities, different from those found by classical fluid dynamical formalisms, by taking into account both many body effects and external force fields.  This is commonly known as (classical) density functional theory (DFT)~\cite{Mermin:1965lo}.  It is able to predict effects driven by the microscale, e.g., the non-smooth droplet profiles which are formed at the gas-liquid-solid trijunction in contact line problems \cite{berim2009simple} and the coexistence of multiple fluid films at critical values of the chemical potential energy in droplet spreading \cite{pereira2012equilibrium}.  It has been used to resolve the paradox of stress and pressure singularities normally found in classical moving contact line problems \cite{sibley2013contact}. What is more, DFT agrees well with molecular dynamics simulations; see, e.g.,~\cite{Lutsko10} and references therein. These advancements motivate more mathematical analysis, in particular, on the well-posedness of the underlying equations being used and on the number and structure of equilibrium states.

As a non-equilibrium extension to DFT for classical fluids, dynamic DFT (DDFT) has been applied to a wide range of problems:
polymeric solutions~\cite{PennaDzubiellaTarazona03},
spinodal decomposition~\cite{ArcherEvans04},
phase separation~\cite{Archer05},
granular dynamics~\cite{MarconiMelchionna07,MarconiTarazonaCecconi07},
nucleation~\cite{vanTeeffelenLikosLowen08},
liquid crystals~\cite{WittkowskiLowenBrand10},
and evaporating films~\cite{ArcherRobbinsThiele10}.
Recently, a stochastic
version of DDFT has been derived~\cite{Lutsko12}, which allows the study of energy barrier crossings, such
as in nucleation. A crucial point is that the computational complexity of DDFT is (essentially) constant in the number of particles, which allows the treatment of macroscopically large systems, whilst retaining microscopic information. Furthermore, due to the universality of the underlying nonlinear, nonlocal partial differential equations, DDFT may be considered as a generalisation of a wider class of such models used in the continuum modelling of many natural phenomena consisting of complex, many body, multi-agent interparticle effects including: pattern formation \cite{camazine2003self}, the flocking of birds, cell proliferation, the self organising of morphogenetic and bacterial species \cite{canizo2010collective, carrillo2009double}, nonlocal reaction-diffusion equations~\cite{al2018dynamical} and even consensus modelling in opinion dynamics\cite{chazelle2017well}. Many of these applications are often described as systems of interacting (Brownian) particles and, in the case of hard particle viscous suspensions, bath-mediated HI effects may be included. 

The HI are forces on the colloids mediated by the bath flow, generated by the motion of the colloidal particles. This in turn produces a nontrivial particle--fluid--particle hydrodynamic phenomenon, the inclusion of which has been shown to have 
substantial effects on the physics of many systems; for example,
they have been found to be the underlying mechanism for the increased viscosity of
suspensions compared to a pure bath~\cite{Einstein06}, the blurring of laning that arises in driven flow~\cite{WysockiLowen11},
the migration of molecules away from a wall~\cite{HodaKumar07}, 
and are particularly complex in confined systems~\cite{happel2012low,LaugaSquires05},
and for active particles and microswimmers, which result in additional HI~\cite{HuberKoehlerYang11}.

Mathematically, the inter-particle forces and HI can be described through the hydrodynamic fields $\varrho$ and $\vec{v}$, the one-body density and one-body velocity fields, respectively. These fields, inherent to a continuum description of a collection of particles, are derived by considering successive moments (density, velocity, heat flux, \dots) of the underlying kinetic system \cite{gorban2014hilbert}. For example, for systems of interacting Newtonian particles, when the momenta are non-negligible, the evolution of the phase space density $f(\vec{r}^N,\vec{p}^N, t)$ for a system of $N$ colloids determining the probability of finding the system in the state $(\vec{r}^N,\vec{p}^N)$ at time $t$ is described by the $N$-body Fokker-Planck equation and the dynamics of the hydrodynamic fields are defined by obtaining closed equations for  $\{\varrho, \varrho\times \vec{v}\} := \int \mathrm{d}\vec{r}^{N-1}\,\mathrm{d}\vec{p}^N\, \{1, \vec{p}/m\}f(\vec{r}^N,\vec{p}^N, t)$, where $m$ is the particle mass.  Here, $\vec{r}^N$ and $\vec{p}^N$ denote the $N$ $d-$dimensional position and momentum vectors of all $N$ particles.

The inclusion of HI leads to a much richer hierachy of fluid equations compared to systems without HI; compare e.g.\ \cite{goddard2012unification} and \cite{archer2009dynamical}. In particular, see e.g. \cite{goddard2012unification}, by integration over all but one particle position, the one-body Fokker-Planck equation may be obtained. If, in addition, two-body HI and interparticle interactions are assumed and the inertia of the colloids is considered small, a high friction limit $\gamma\to \infty$ may be taken~\cite{goddard2012overdamped}.  The result is that the velocity distribution converges to a Maxwellian, and one can eliminate the momentum variable through an adiabatic elimination process that is based on multiscale analysis \cite{pavliotis2008multiscale}. The final one-body Smoluchowski equation for $\varrho$ is a novel, nonlinear, nonlocal PDE shown to be independent of the unknown kinetic pressure term $\int \mathrm{d}\vec{r}\,\mathrm{d}\vec{p}\, m^{-2} \vec{p}\otimes\vec{p} f(\vec{r},\vec{p}, t)$, which normally persists at $\gamma = O(1)$ (see \cite{goddard2012overdamped}, Theorem 4.1). 

Existence, uniqueness and global asymptotic stability of the novel Smoluchowski equation in this overdamped limit has, until this work, remained unproven. The inclusion of HI is interesting from both physical and mathematical standpoints. Physically, as above, the HI give rise to a much more complex evolution in the density.  Mathematically, the convergence to equilibrium will depend inherently on the spectral properties of the effective diffusion tensor and effective drift vector arising from the HI. What is more, since the full $N$-body Fokker-Planck equation is a PDE in a very high dimensional phase space, well-posed nonlinear, nonlocal PDEs governing the evolution of the one-particle distribution function, valid in the mean field limit, describing the flow of nonhomogeneous fluids are desirable for computational reasons.

The equations studied in this paper are related to the McKean-Vlasov equation \cite{chayes2010mckean}, a nonlinear nonlocal PDE of Fokker-Planck type that arises in the meanfield limit of weakly interacting diffusions. The novelty of the present problem lies in the spatially dependent nonhomogeneous diffusion and advection tensors as well as the nonlinear, nonlocal boundary condition. We are interested in global existence, uniqueness, positivity and regularity of the density, including the nonlocal effects due to HIs, which represents a non-trivial deviation from classical McKean-Vlasov equation. Primarily used in modelling the evolution of the density of plasma, the McKean-Vlasov equation has received much attention in recent years within a rigorous mathematical framework, particularly with the use of calculus of variations techniques by writing the evolution equation for the density in a gradient flow form with a suitably defined free energy functional. Previous well-posedness studies of similar nonlinear, nonlocal PDEs focused on periodic boundary conditions and without nonlocal effects; see, e.g.,~\cite{chazelle2017well, greg_mckean_vlasov_torus}. In contrast, we are interested in the well-posedness subject to no-flux boundary conditions. This choice admits the nontrivial effect of the two body forces generated by the potential $V_2$ interacting with density on the boundary of the physical domain. The motivation for this choice of boundary condition is physical; it corresponds to a closed system of particles in which the particle number is conserved over time. It is clear that most applications of such equations will be in confined systems, rather than  a periodic domain and, as such, no-flux boundary conditions are natural. 

\subsection{Main Results.}
The main results of this work are threefold. 
\begin{enumerate}
\item We establish the existence and uniqueness of the transient density and flux, $\varrho(\vec{r},t)$, $\vec{a}(\vec{r},t)$ respectively, in a weak sense, for two-body HI (see Theorem \ref{thm:exis_uniq_rho_and_a}).
\item We establish a gradient flow structure and Lyapunov stability for classical solutions $\varrho(\vec{r},t)$ subject to two-body HI (see Theorem \ref{thm:eig_expan_gradient_flow_lyapunov}).
\item We derive an \emph{a priori} convergence estimate of the density classical density $\varrho(\vec{r},t)$ to equilibrium in $L^2$, also subject to two-body HI (see Theorem \ref{thm:apriori_conv_est}).
\end{enumerate}

The novelty of this work is shaped by the non-trivial technical difficulties arising from the inclusion of the HI tensors, which manifests as additional nonlocality in the evolution equation for the density and flux, as well as imposing a nonlinear, nonlocal no-flux boundary condition in order to satisfy conservation of mass. These results are of particular interest for the rigorous mathematical foundations of computational models of colloidal fluids. 

\subsection{Organisation of the Paper.}
The paper is organised as follows: in Section \ref{sec:preliminaries} we state the main dynamic and equilibrium equations, present a free energy framework, introduce the HI operators, and present a recasting of the main equations as a gradient flow structure. In Section \ref{sec:assumptions_definitions} set forth the assumptions and regularity of the potentials and HI tensors for the dynamic and equilibrium equations, and define a weak formulation for the density and flux. In Section \ref{sec:statement_of_main_results} we state the main results of the present work in a precise manner. In Section \ref{sec:existence_uniqueness_denstiy_flux_pair} we prove an existence and uniqueness theorem the weak density $\varrho(\vec{r},t)$ and flux $\vec{a}(\vec{r},  t)$. In Section \ref{sec:behaviour_classical_solutions} we rigorously connect classical densities to a variational principle involving a unique free energy functional $\mathcal{F}[\varrho]$, prove an H-Theorem for the classical density, and characterise equilibrium densities and derive a convergence of $\varrho(\vec{r},t)$ to equilibrium in $L^2$ as $t \to \infty$. Finally, in Section \ref{sec:discussion}, we present our concluding remarks and state some open problems. 

\section{Preliminaries}\label{sec:preliminaries}
Let $\Omega$ be a bounded subset of $\mathbb{R}^d$ with a $C^1$ boundary $\partial\Omega\in\mathbb{R}^{d-1}$ and consider the following coupled system of non-linear and non-local PDEs, for the unknown density $\varrho : \Omega \times [0,\infty) \to \mathbb{R}$ and flux  $\vec{a}:\Omega \times [0,\infty) \to \mathbb{R}^d$, evolving according to 
\begin{subequations}
\begin{align}
&\partial_{t}\varrho(\vec{r},t) +\nabla \cdot\vec{a}(\vec{r}, t)=0,\label{eq:ddft_eq_dyn_rho}  \quad \quad \quad\quad\quad\quad \qquad 
\qquad \qquad \qquad \qquad  \vec{r} \in \Omega \subset \mathbb{R}^d, \, t \in[0,T]\\ 
&\left(\bm{1}+\left(\bm{Z}_1\star \varrho\right) (\vec{r},t)\right) \,\vec{a}(\vec{r},  t)+\varrho(\vec{r},t)\left(\bm{Z}_2\star \vec{a}\right) (\vec{r},t)=\vec{f}(\vec{r},\varrho, t) ,\label{eq:ddft_eq_dyn_a}  \,\,\, \,\vec{r} \in \Omega \subset \mathbb{R}^d, \, t \in[0,T]
\end{align}
\end{subequations}
subject to the following boundary and initial conditions
\begin{subequations}
\begin{align}
&\vec{a}(\vec{r}, t)\Big|_{\partial \Omega}\cdot\vec{n} = 0,\label{bc:DDFT} \qquad \,\,\,\, t \in [0,T]\\
& \varrho(\vec{r}, 0) = \varrho_0(\vec{r}), \,\,\,\qquad\quad \vec{r}\in \Omega \subset \mathbb{R}^d \label{eq:init_rho0}\\
& \vec{a}(\vec{r}, 0) = \vec{a}_0(\vec{r}), \,\,\,\quad \quad \quad \vec{r}\in \Omega \subset \mathbb{R}^d\label{eq:init_a0}
\end{align}
\end{subequations}
where $T>0$ is arbitrary, $\star$ denotes convolution in space and $\vec{n}$ is a unit normal vector pointing out of the domain $\Omega$. The boundary condition \eqref{bc:DDFT} is a nonlinear Robin condition imposing that the flux through the boundary $\partial \Omega$ is zero. Note that by the divergence theorem, the boundary condition \eqref{bc:DDFT} combined with equation \eqref{eq:ddft_eq_dyn_rho} implies that $\int \mathrm{d}\vec{r}\, \varrho(\vec{r},t) = M>0$, for all time, where $M$ is a constant independent of time. Hence \eqref{bc:DDFT} as a boundary condition for $\vec{a}(\vec{r},t)$ also provides an integral condition for $\varrho(\vec{r},t)$ (i.e., conserved mass), meaning, \eqref{eq:ddft_eq_dyn_rho}--\eqref{eq:ddft_eq_dyn_a} has two auxiliary conditions for two unknowns. Without loss of generality we let $M = 1$. Moreover, $\bm{Z}_1: \Omega^2 \to \mathbb{R}^{d\times d}$ and  $\bm{Z}_2: \Omega^2 \to \mathbb{R}^{d\times d}$ are the so-called Hydrodynamic Interaction (HI) tensors, and $\vec{f}: \Omega\times \mathbb{R} \times [0,\infty)\to \mathbb{R}^d$ is a body force given by
\begin{align}
\vec{f}(\vec{r}, \varrho, t):=\left(\nabla  + \nabla V_1(\vec{r},t)+ \left( \nabla V_{2}\star \varrho\right) (\vec{r},t)\right)\varrho(\vec{r},t),\label{eq:body_force_f}
\end{align}
where $V_1 : \Omega \times [0,\infty) \to \mathbb{R}, \quad V_2 : \Omega^2 \to \mathbb{R}$ are, one and two-body functions respectively. 
  
We will also consider the stationary problem associated with  \eqref{eq:ddft_eq_dyn_rho} - \eqref{bc:DDFT}, for the equilibrium density $\varrho_0:\Omega \to \mathbb{R}$ and equilibrium flux $\vec{a}_0 :\Omega\times\mathbb{R} \to \mathbb{R}^d$ governed by    
\begin{subequations}
\begin{align}
&\nabla \cdot\vec{a}_0(\vec{r})=0, \qquad \qquad \qquad \qquad\qquad 
\qquad \qquad \qquad \qquad \qquad  \qquad \,\,\vec{r} \in \Omega \subset \mathbb{R}^d,\label{eq:ddft_rho_stationary}\\ 
&\left(\bm{1}+\left(\bm{Z}_1\star \varrho_0\right) (\vec{r})\right) \,\vec{a}_0(\vec{r})+\varrho_0(\vec{r})\left(\bm{Z}_2\star \vec{a}_0\right) (\vec{r})=\vec{f}^{eq}(\vec{r},\varrho_0), \quad\qquad   \vec{r} \in \Omega \subset \mathbb{R}^d,\label{eq:ddft_a_stationary}
\end{align}
\end{subequations}
subject to the boundary condition
\begin{align}
\vec{a}_0(\vec{r})\Big|_{\partial \Omega}\cdot\vec{n} &= 0,
\end{align}
and where $\vec{f}^{eq}:\Omega\times \mathbb{R}\to\mathbb{R}^d$ is defined as
\begin{align}
\vec{f}^{eq}(\vec{r},\varrho_0):=\left(\nabla  + \nabla V_1^{eq}(\vec{r})+ \left( \nabla V_{2}^{eq}\star \varrho_0\right) (\vec{r},t)\right)\varrho_0(\vec{r}),
\end{align}
for $V_1^{eq}:\Omega\to\mathbb{R}$ and $V_2^{eq}:\Omega^2\to\mathbb{R}$. Additionally, for each $\phi:\Omega \times [0,\infty)\to \mathbb{R}^{+}$ we define the nonhomogeneous diffusion tensor $\bm{D}_\phi:\Omega\times[0,\infty)\to \mathbb{R}^{d\times d}$ as
\begin{align}
\bm{D}_{\phi}(\vec{r},t):=\Big(\bm{1}+\int\mathrm{d}\vec{r}'
\,\phi(\vec{r}',t)\bm{Z}_1(\vec{r}-\vec{r}')\Big)^{-1}.\label{eq:def_diffusion_tensor}
\end{align}
Note that, for every fixed $\vec{r}, \phi$ and $t$, $\bm{D}(\vec{r},t)$ is real and symmetric and was previously proven to be strictly positive definite (and therefore invertible) (see \cite[Corollary 4.2]{goddard2012overdamped} for details). Therefore, $\bm{D}_{\phi}^{-1}(\vec{r},[\varrho],t)$ exists and is uniformly bounded in space and time for each $\phi$, and, as such we carry this over to Assumption \eqref{ass:D_pos_def_weak_diffable}. Throughout this work we denote the eigenvalues of $\bm{D}_\phi$ as $\mu_i^\phi$ for $i=1,\cdots, d$ and the smallest and largest eigenvalues are denoted $\mu_{\min}^{\phi}$ and $\mu_{\max}^\phi$, respectively. Note that for each $i=1,\cdots,d$, $\mu_{i}^\phi$ are parametrised by time $t\in [0,\infty)$. The nonhomogeneous advection tensor $\bm{A}:\Omega\times \mathbb{R}^d\times [0,\infty)\to \mathbb{R}^{d\times d}$ is an effective background flow induced by the hydrodynamic interactions defined by the action
\begin{align}\label{eq:def_of_V1_eff}
\bm{A}(\vec{r}, [\vec{a}], t):=\int \mathrm{d}\vec{r}'\,\bm{Z}_2(\vec{r}-\vec{r}')\vec{a}(\vec{r'},t)
\end{align}
where $[\cdot]$ denotes functional dependence (i.e. for any function $f$, we write $f([\vec{a}])$ to express dependence of $f$ on the whole function $\vec{a}$). For later calculations it is convenient to define the generalised Hydrodynamic Interaction operator $\mathcal{H}_{\phi]}:L^{2}(\Omega)\times L^{2}(\Omega) \to L^{2}(\Omega)$ which for each $\phi:\Omega\times [0,\infty)\to \mathbb{R}^+$ acts on a vector field $\vec{v}:\Omega\times[0,t)\to \mathbb{R}^d$ as
\begin{align}
\mathcal{H}_{\phi}\vec{v}(\vec{r},t) &:= \bm{D}_\phi^{-1}\vec{v}(\vec{r},t) +\phi \left( \bm{Z}_2 \star \vec{v}\right)(\vec{r},t). \label{def:H_phi_op}
\end{align}
The operator $\mathcal{H}_{\phi}$ will be useful for writing the dynamical equations \eqref{eq:ddft_eq_dyn_rho}--\eqref{eq:ddft_eq_dyn_a} in weak form for the results of Section \ref{sec:existence_uniqueness_denstiy_flux_pair} as well as establishing results on classical solutions in gradient flow form for Section \ref{sec:behaviour_classical_solutions}. 

Related to the system \eqref{eq:ddft_eq_dyn_rho}-\eqref{eq:ddft_eq_dyn_a}, we present the free energy functional $\mathcal{F}:P^+_{\text{ac}}(\Omega)\to \mathbb{R}$ where $P^+_{\text{ac}}$ is the set of strictly positive definite absolutely continuous probability measures on $\Omega\subset\mathbb{R}^d$. We define
\begin{align}
\mathcal{F}[\varrho](t)&:=\int \mathrm{d}\vec{r}\,\varrho(\vec{r},t)\left( \log \varrho(\vec{r},t) -1 \right)+\int \mathrm{d}\vec{r}\,\varrho(\vec{r},t)\,V_1(\vec{r},t)+\frac{1}{2} \int \mathrm{d}\vec{r}\,\varrho(\vec{r},t)\,\left( V_2\star\varrho\right) (\vec{r},t).\label{eq:def_of_F}
\end{align}
Finally, we observe that the right hand side of \eqref{eq:ddft_eq_dyn_a} resembles $\varrho\nabla \delta \mathcal{F}/(\delta \varrho)$ where $\delta/(\delta \varrho)$ denotes the functional derivative with respect to $\varrho$ given by
\begin{align}
\frac{\delta \mathcal{F}}{\delta \varrho}[\varrho]:= \lim_{\epsilon\to 0} \frac{\mathcal{F}[\varrho+\epsilon\omega]-\mathcal{F}[\varrho]}{\epsilon}
 \end{align} 
where $\omega\in P^+_{\text{ac}}(\Omega)$ is an arbitrary function and $\epsilon\omega$ is the variation of $\varrho$. For the equilibrium problem \eqref{eq:ddft_rho_stationary}-\eqref{eq:ddft_a_stationary}, note that for a given a finite flux vector $\vec{a}_0$ with $\nabla\cdot \vec{a}_0 = 0$, it is not obvious that $\varrho_0$ is necessarily a minimiser of $\mathcal{F}$. However indeed, for the particular choice $\vec{a}_0\equiv \vec{0}$ (which is a natural and physically realistic solution), $\varrho_0$ is necessarily a extremum of $\mathcal{F}$, by direct verification. We will in fact show under reasonable assumptions, in Section \ref{sec:behaviour_classical_solutions}, that the classical solutions $\varrho_0$ and $\vec{a}_0 \equiv \vec{0}$ to \eqref{eq:ddft_rho_stationary}-\eqref{eq:ddft_a_stationary} with $\varrho_0$ being an extremum of $\mathcal{F}$ are the unique equilibrium solutions. Additionally, in Proposition \ref{prop:association _of_free_energy} the functional $\mathcal{F}$ will be shown to uniquely define the density field minimising the free energy associated to the system \eqref{eq:ddft_eq_dyn_rho}--\eqref{eq:ddft_eq_dyn_a}  as $t\to \infty$. 
 
\subsection{Gradient and Non-Gradient Forms.}

In this section we present two alternate forms of the DDFT \eqref{eq:ddft_eq_dyn_rho}--\eqref{eq:ddft_eq_dyn_a}. The first recasting is useful for analysing the behaviour of classical solutions $\varrho(\vec{r},t)$ and $\vec{a}(\vec{r},t)$, in particular to establish the Lyapunov stability of $\varrho(\vec{r},t)$ with respect to $\mathcal{F}$. The second recasting is useful for establishing existence and uniqueness of the weak density $\varrho(\vec{r},t)$ and weak flux $\vec{a}(\vec{r},  t)$ solving a weak formulation of \eqref{eq:ddft_eq_dyn_rho}--\eqref{eq:ddft_eq_dyn_a}. First we present the gradient flow structure. We will see that, under suitable conditions, (as shown in Corollary \ref{cor:gradient_flow_structure}) and with the \eqref{def:H_phi_op}, we may rewrite \eqref{eq:ddft_eq_dyn_rho}, \eqref{eq:ddft_eq_dyn_a} as
\begin{align}\label{eq:ddft-eq-gradient_flow_form}
\partial_{t}\varrho = \nabla \cdot \left(\mathcal{H}^{-1}_{\varrho}\varrho \nabla \tfrac{\delta \mathcal{F}}{\delta \varrho}[\varrho]\right),
\end{align}
From \eqref{eq:ddft-eq-gradient_flow_form} we see that full the dynamics has an implicit gradient flow structure. The invertibility of $\mathcal{H}_{\varrho}$ is established by Proposition \ref{thm:cond_converg_fred_det}. 

We now present the non-gradient flow structure for the well-posedness results. Due to the non-explicit form of $\mathcal{H}^{-1}_{\varrho}$, \eqref{eq:ddft-eq-gradient_flow_form} is not a useful form in establishing the well-posendess. For these results we use the following explicit (but non-gradient) form of the evolution equations \eqref{eq:ddft_eq_dyn_rho}, \eqref{eq:ddft_eq_dyn_a}
\begin{align}
&\partial_t\varrho=\,\nabla \cdot\left(\bm{D}_{\varrho}\nabla \varrho+\varrho\,\bm{D}_{\varrho}\left(\nabla  ( V_1+ V_2\star\varrho) + \bm{A}\right)\right),\label{eq:ddft_dyn_non_gradient_flow}
\\
&\bm{D}_{\varrho}\,\left(\nabla \varrho + 
\varrho(  \nabla V_1(\vec{r},t)+  \nabla V_2\star \varrho+\bm{A})\right)\cdot\vec{n}\big|_{\partial \Omega} = 0. \label{eq:ddft_dyn_non_gradient_flow_bc}
\end{align}
We continue to the next section by providing the assumptions and regularity of the potentials, diffusion and advection tensors and initial data. 
%

\section{Main Assumptions, Notation \& Definitions.}\label{sec:assumptions_definitions}
We have the following assumptions for the evolution problem \eqref{eq:ddft_dyn_non_gradient_flow}.


\subsection{Assumptions D}

For the dynamics we assume:
\begin{itemize}\label{assumptions:D_V1_V2_rho0}
\item For $\phi:\Omega\times[0,\infty)\to\mathbb{R}^{+}$, the diffusion tensor $\bm{D}_{\phi}$ is symmetric, positive definite. Addiitionally, the first derivatives of $\left(\bm{D}_{\phi}\right)_{ij}$ and $\bm{A}_{ij}$ are bounded in $L^\infty(\Omega)$ \begin{align}\label{ass:D_pos_def_weak_diffable}
\left(\bm{D}_{\phi}\right)_{ij}, \, \bm{A}_{ij}\in W^{1,\infty}(\Omega). \tag{D1}
\end{align}
Additionally, by the positive definite property in \eqref{ass:D_pos_def_weak_diffable} we may uniquely define the square root of $\bm{D}_{\phi}$ denoted $\bm{D}_{\phi}^{1/2}$ such that $\bm{D}_{\phi}^{1/2}\bm{D}_{\phi}^{1/2}=\bm{D}_{\phi}$ for every $\vec{r}\in \Omega$, $t\in [0,T]$ and each $\phi$. 

\item The diagonal and off-diagonal blocks of the HI tensors are real, symmetric and uniformly bounded in the sense
\begin{align}\label{ass:Z2_uniformly_bd}
\|\bm{Z}_2\|_{L^\infty(\Omega)}<\infty, \quad  \|\bm{Z}_1\|_{L^\infty(\Omega)}<\infty.\tag{D2}
\end{align}



\item We assume $\varrho_0(\vec{r}), \,\vec{a}_0(\vec{r}) \in C^\infty(\Omega)$. Additionally, the initial data $\varrho_0$ is a non-negative,  absolutely continuous probability density
\begin{align}\label{ass:rho_in_L2_P_ac}
\varrho_0\in P_{ac}(\Omega)\cap C^\infty(\Omega). \tag{D3}
\end{align}
The smoothness restrictions may be weakened to square-integrable, admitting a wider class of initial data once the first well-posedeness result is obtained.
At the end of Section \ref{sec:behaviour_classical_solutions} we will show by simple application of Harnack's inequality, that we obtain strictly positive densities $\varrho(\vec{r},t_1)>0$ after an arbitrarily small time $t_1>0$. Principally this is provided by the property \eqref{ass:D_pos_def_weak_diffable}, since $\bm{D}_{\phi}$ is positive definite, the diffusion of density in the system \eqref{eq:ddft_dyn_non_gradient_flow} is everywhere propagating in $\Omega$. 

\item The potentials each have two bounded derivatives, 
\begin{align}\label{ass:V1_V2_in_W_1_inf}
 V_1(\cdot, t),V_2\in W^{2,\infty}(\Omega), \tag{D4}
\end{align}
for each $t\in [0,T]$. The functions $V_1$ and $V_2$ are the confining and two-body interaction potentials respectively.
\end{itemize}

\subsection{Assumptions E}
For the equilibrium problem \eqref{eq:ddft_rho_stationary}, \eqref{eq:ddft_a_stationary} we will assume:
\begin{itemize}
\item The potentials have first order weak derivatives in $L^2(\Omega)$ 
\begin{align}\label{ass:V1_V2_in_H_1}
 V_1^{eq},V_2^{eq}\in H^{1}(\Omega).\tag{E1}
\end{align}
\end{itemize}
In particular, Assumption \eqref{ass:V1_V2_in_H_1} will permit us to establish smooth stationary densities. 
In general we admit non-convex $V_1$ and $V_2$, for example multi-well potentials. 


\subsection{Weak Formulation.}
 We now define a weak formulation of the dynamics \eqref{eq:ddft_eq_dyn_rho}--\eqref{eq:ddft_eq_dyn_a}.  
 
\begin{definition}[Weak Solution For Density-Flux Pair.]
We now consider the weak formulation of \eqref{eq:rho_it_eqn}--\eqref{eq:a_it_eqn}. Let $\vec{f}_{\varrho}: = \vec{f}(\vec{r}, \varrho, t)$. We say that $\varrho\in L^2([0,T]; H^1(\Omega))\cap L^\infty([0,T]; L^2(\Omega))$, $\partial_t\varrho\in L^2([0,T]; H^{-1}(\Omega))$ and $\vec{a}\in L^2([0,T];L^2(\Omega,\varrho^{-1}))$ is a weak solution to \eqref{eq:ddft_eq_dyn_rho}--\eqref{eq:ddft_eq_dyn_a} if for every $\eta\in L^2([0,T]; H^1(\Omega))$ and $\vec{v}\in L^2([0,T]; L^2(\Omega))$
\begin{subequations}{}
\begin{align}
&\int_0^T\mathrm{d}t\, \langle \partial_t \varrho(t), \, \eta(t)  \rangle +\int_0^T\mathrm{d}t\, \int  \mathrm{d}\vec{r}\, \nabla \eta\cdot \bm{D}_{\varrho}\,\left(\nabla \varrho
 +\varrho\,\nabla ( V_1+ V_2\star\varrho) + \varrho\bm{Z}_2\star \vec{a}\right) =0,\label{eq:weak_formulation_pair_rho}\\
 & \int_0^T\mathrm{d}t\,  \int  \mathrm{d}\vec{r}\,\vec{v}\cdot \mathcal{H}_{\varrho}\vec{a} = \int_0^T\mathrm{d}t\,  \int  \mathrm{d}\vec{r}\, \vec{v}\cdot\vec{f}_{\varrho},\label{eq:weak_formulation_pair_a}
\end{align}
\end{subequations}
for the initial data $\varrho(\vec{r},0) = \varrho_0(\vec{r})\in C^\infty(\Omega)$ and $\vec{a}_0(\vec{r},0) = \vec{a}_0(\vec{r})\in C^\infty(\Omega)$. 
\end{definition}


\section{Statement of Main Results}\label{sec:statement_of_main_results}

Our main result concerns existence and uniqueness of $\varrho(\vec{r},t)$ and $\vec{a}(\vec{r},t)$ solving \eqref{eq:ddft_eq_dyn_rho}--\eqref{eq:ddft_eq_dyn_a}.
\\
\begin{theorem}[Existence \& Uniqueness of Weak $(\varrho(\vec{r},t),\, \vec{a}(\vec{r}, {[\varrho]}, t) )$ with Two-Body HI]\label{thm:exis_uniq_rho_and_a}
Let $\|\mu_{\max}\|_{L^\infty([0,T])}\|\bm{Z}_2\|_{L^\infty(\Omega)}<1$, and let  $\varrho_0 \in C^\infty(\Omega)$, $\varrho\geq 0$ and $\int  \mathrm{d}\vec{r} \varrho_0(\vec{r})=1$. Then there exists a unique weak solution $\varrho\in L^\infty([0,T];L^2(\Omega))\cap L^2([0,T]; H^1(\Omega))$, with
$\partial_t\varrho\in L^2([0,T]; H^{-1}(\Omega))$ to \eqref{eq:weak_formulation_pair_rho} as well as a unique flux $\vec{a}(\vec{r},  t)\in L^2([0,T]; L^2(\Omega))$ solving \eqref{eq:weak_formulation_pair_a}. 

Additionally, for $\varrho(\vec{r},t)$ the following energy estimate holds
\begin{align}
\|\varrho\|_{L^{\infty}([0,T];L^2(\Omega))}+\|\varrho\|_{L^{2}([0,T];H^1(\Omega))}+\|\partial_t\varrho\|_{L^2([0,T]; H^{-1}(\Omega))}\leq C(T) \|\varrho_{0}\|_
{L^2(\Omega)},
\end{align}
where $C(T)$ is a constant dependent on on the finite time $T$. Meanwhile, \\ $\vec{a}(\vec{r},  t)\in L^2([0,T]; L^2(\Omega))$  is a unique weak solution solving \eqref{eq:weak_formulation_pair_a} if and only if $\vec{a}(\vec{r},  t)$ is a minimiser of $\mathcal{J}:L^2([0,T]; L^2(\Omega))\to \mathbb{R}$ given by
\begin{align}
\mathcal{J}[\vec{v}] := \frac{1}{2}\int_0^T\mathrm{d}t\,  \int  \mathrm{d}\vec{r}\,\varrho(\vec{r},t)^{-1}\vec{v}(\vec{r},t)\cdot\left(\mathcal{H}_{\varrho}\vec{v}\right)(\vec{r},t)  -\int_0^T\mathrm{d}t\,  \int  \mathrm{d}\vec{r}\,\vec{v}(\vec{r},t)\cdot\nabla \frac{\delta\mathcal{F}}{\delta\varrho}[\varrho](\vec{r},t).
\end{align}
\end{theorem}
The second main result of this paper concerns classical solutions of \eqref{eq:ddft_eq_dyn_rho}--\eqref{eq:ddft_eq_dyn_a}.

\begin{theorem}[Eigenfunction Expansion, Gradient Flow Structure, and Lyapunov Stability For Classical Solutions]\label{thm:eig_expan_gradient_flow_lyapunov}{\ \\}
Let $\varrho(\vec{r},t)\in C^1([0,\infty);C^2(\Omega))$ and $\vec{a}(\vec{r},t)\in C([0,\infty);C^{2d}(\Omega))$ be classical solutions to \eqref{eq:ddft_eq_dyn_rho}--\eqref{eq:ddft_eq_dyn_a} and suppose $\|\mu_{\max}^{\varrho}\|_{L^\infty([0,T])}\|\bm{Z}_2\|_{L^\infty}<1$. Then,
\begin{enumerate}
\item For the given density field $\varrho(\vec{r},t)>0$ there exists a unique classical flux $\vec{a}(\vec{r}, t)$ solving \eqref{eq:ddft_eq_dyn_a} of the form
\begin{align}
\vec{a}(\vec{r},t) = \sum_{n\in\mathbb{N}}\gamma_n^{-1}\langle\vec{u}_n,\nabla \frac{\delta\mathcal{F}}{\delta\varrho}[\varrho]\rangle_{L^2(\Omega)} \vec{u}_n(\vec{r},t),
\end{align}
where $\gamma_n$ and $\vec{u}_n(\vec{r},t)$ are the eigenvalues and eigenvectors of $\mathcal{H}_\varrho$, a real, positive definite and self-adjoint operator in $L^2([0,T];L^2(\Omega,\varrho^{-1}))$.
\item The classical density satisfies the nonlinear-nonlocal PDE, of implicit gradient flow form
\begin{align}
\partial_{t}\varrho = \nabla \cdot \left(\mathcal{H}^{-1}_{\varrho}\varrho \nabla \tfrac{\delta \mathcal{F}}{\delta \varrho}[\varrho]\right).
\end{align}
additionally, $\varrho(\vec{r},t)>0$ for all $\vec{r}\in\Omega$ and $t\in [0,\infty)$ and is Lyapunov stable with respect to the free energy $\mathcal{F}$ such that, for all $t\in [0,\infty)$
\begin{align}
\der[]{t}\mathcal{F}[\varrho] = -\| \mathcal{H}^{-1/2}_{\varrho} \varrho \nabla \frac{\delta\mathcal{F}}{\delta\varrho}[\varrho]\|^2_{L^2_{\varrho ^{-1}}}\leq 0.
\end{align}
\item The steady flux solving \eqref{eq:ddft_rho_stationary}--\eqref{eq:ddft_a_stationary} is given by $\vec{a}_0(\vec{r})\, != \vec{0}$, and hence all steady states $\varrho_0(\vec{r})$ solving \eqref{eq:ddft_a_stationary} are equilibria.
\item $\vec{a}_0(\vec{r}) = 0$ for all $\vec{r}\in \Omega$ iff $\varrho_0(\vec{r})$ is a critical point of the free energy $\mathcal{F}$.
\end{enumerate}

\end{theorem}

The third main result of this paper concerns \emph{a priori} estimates for exponential convergence of the density to equilibrium in $L^2$.
\\
\begin{theorem}[A Priori Convergence Estimates]\label{thm:apriori_conv_est}
\par
Let $\varrho\in C^1([0,\infty];C^2(\Omega)) $ be a solution of \eqref{eq:ddft_dyn_non_gradient_flow} with initial data $\varrho_0\in L^2(\Omega)$ a probability density. Let $1 \leq 1/4\|V_2\|_{L^\infty}^{-1}$, so that $\varrho_\infty$ is the unique equilibrium density ensured by Lemma \ref{thm:exis_fix_point}. If,
\begin{align}
r_{t} &:= \hat{\mu}_{\min}(t)c^{-2}_{pw}- 2\hat{\mu}_{\max}(t) \left(\|\nabla V_1^{eq}\|_{L^\infty(\Omega)}^2+ (e+1)\|\nabla V_2^{eq}\|_{L^{\infty}(\Omega)}^2\right) \\
&\quad - \hat{\mu}_{\max}(t)\|\bm{Z}_2\|_{L^\infty(\Omega)}^2\|\vec{a}\|_{L^2([0,T];L^1(\Omega))}^2>0,
\end{align} 
where $\hat{\mu}_{\max}(t) = \int_0^t\mathrm{d}s\,\mu_{\max}(s)$ and $\hat{\mu}_{\min}(t) = \int_0^t\mathrm{d}s\,\mu_{\min}(s)$ are the time-mean of the largest and smallest eigenvalues of $\bm{D}_{\varrho(\vec{r},t)}$. Then, $\varrho\to \varrho_\infty\in C^\infty(\Omega)\cap P_{ac}^{+}(\Omega)$ in $L^2(\Omega)$ exponentially as $t\to \infty$.  In particular the convergence in $L^2(\Omega)$ is given by
\begin{align}
\|\varrho(\cdot, t)-\varrho_\infty(\cdot)\|^2_{L^2(\Omega)}\leq \|\varrho_0(\cdot)-\varrho_\infty(\cdot)\|_{L^2(\Omega)}^2 e^{-r_{ t}}
\end{align}
as $t \to \infty$. The constant $c_{pw}$, is the Poincar{\'e}-Wirtinger constant, for which the optimal value is the inverse square root of the smallest eigenvalue of the Laplacian on the domain $\Omega$ with no-flux boundary conditions.
\end{theorem}

We now give our arguments for Theorem \ref{thm:exis_uniq_rho_and_a}.

\section{Existence \& Uniqueness of the Weak Density-Flux Pair}
\label{sec:existence_uniqueness_denstiy_flux_pair}

In this section we determine the existence and uniqueness of the density $\varrho(\vec{r},t)$ and $\vec{a}(\vec{r},t)$ solving \eqref{eq:ddft_eq_dyn_rho}-\eqref{eq:ddft_eq_dyn_a}, in the sense \eqref{eq:weak_formulation_pair_rho}--\eqref{eq:weak_formulation_pair_a}. We postpone some calculations, e.g., estimates of $\varrho$ in various energy norms, to Appendix \ref{app:results_ex_uni} and the explicit calculations can be found in \cite{rdmwthesisddft}. For $\varrho(\vec{r},t)$, the method we use  follows \cite{chazelle2017well}, i.e., we prove the existence and uniqueness of the weak density by defining a sequence of linear parabolic equations, whose solutions converge strongly to a function $\varrho(\vec{r},t)$ in $L^1([0,T];L^1(\Omega))$ that solves a weak formulation \eqref{eq:weak_formulation_pair_rho}, but here we must account for the confining potential $V_1$, the effective drift $\bm{A}$, and the diffusion tensor $\bm{D}_\phi$. For $\vec{a}(\vec{r},t)$ we perform an iterated argument based on the Lax-Milgram Theorem for a sequence of bilinear forms on the flux, again converging strongly to a function $\vec{a}(\vec{r},t)$ in $L^1([0,T];L^1(\Omega))$ solving the weak formulation  \eqref{eq:weak_formulation_pair_a}. 

To establish existence of weak solutions $\left(\varrho,\vec{a}\right)$ in the sense \eqref{eq:weak_formulation_pair_rho}--\eqref{eq:weak_formulation_pair_a} we consider a frozen version of \eqref{eq:ddft_eq_dyn_rho}--\eqref{eq:ddft_eq_dyn_a}. 
We let $(\varrho_{n-1},\vec{a}_{n-1})$ be a given density-flux pair and consider the following frozen problem
\begin{subequations}{}
\begin{align}
\partial_t\varrho_n &= \nabla\cdot \left(\bm{D}_{\phi}\left(\nabla\varrho_n + \varrho_n\nabla V_1+\varrho_n\nabla V_2\star \varrho_{n-1} + \varrho_{n}\bm{Z}_2\star \vec{a}_{n-1}\right)  \right),\label{eq:rho_it_eqn}\\
\mathcal{H}_{\phi}\vec{a}_n &= \vec{f}(\varrho_{n},\varrho_{n-1}),\label{eq:a_it_eqn}
\end{align}
\end{subequations}
where $\vec{f}(\varrho_{n},\varrho_{n-1}) := \nabla\varrho_{n}+ \varrho_n\nabla V_1+\varrho_n\nabla V_2\star \varrho_{n-1}$. First, by classical theory (e.g., see Appendix \ref{thm:classical_PDE_existence}, or \cite{evans2002partial}), for equation \eqref{eq:rho_it_eqn}, there exists a unique $\varrho_n\in C^\infty([0,T]; C^\infty(\Omega))$ for each $n\in \mathbb{N}$. Second, for equation \eqref{eq:a_it_eqn} there exists a unique $\vec{a}_n\in L^2\left(\Omega,\phi^{-1}\right)$. We now consider the following weak formulation of the frozen problem. Find $\varrho_n\in L^2([0,T];L^2(\Omega))$ and $\vec{a}_n\in L^2([0,T];L^2(\Omega))$ such that
\begin{subequations}{}
\begin{align}
&\int_0^T\mathrm{d}t\, \langle \partial_t \varrho_n(t), \, \eta(t)  \rangle +\int_0^T\mathrm{d}t\, \int  \mathrm{d}\vec{r}\, \nabla \eta\cdot \bm{D}_{\phi}\,\left(\nabla \varrho_n
 +\varrho_n\,\nabla ( V_1+ V_2\star\varrho_{n-1})\right.\nonumber\\
 &\qquad\qquad\qquad\qquad\qquad\qquad\qquad\qquad\qquad\qquad\qquad\qquad \left. + \varrho_n\bm{Z}_2\star \vec{a}_{n-1}\right) =0,\label{eq:weak_formulation_pair_frozen_rho}\\
 & \int_0^T\mathrm{d}t\,  \int  \mathrm{d}\vec{r}\,\vec{v}\cdot \mathcal{H}_{\phi}\vec{a}_n = \int_0^T\mathrm{d}t\,  \int  \mathrm{d}\vec{r}\, \vec{v}\cdot\vec{f}(\varrho_n,\varrho_{n-1}).\label{eq:weak_formulation_pair_frozen_a}
\end{align}
\end{subequations}
subject to the conditions
\begin{subequations}
\begin{align}
\int\mathrm{d}\vec{r}\, \varrho_n(\vec{r},t) &= 1,\\
\vec{a}_n(\vec{r},t)\big|_{\partial\Omega}\cdot\vec{n} &= 0,
\end{align}
\end{subequations}
for every $\eta\in L^2([0,T];H^1(\Omega))$ and $\vec{v}\in L^2([0,T];L^2(\Omega))$. We will make use of the following bilinear and linear functionals.

\begin{definition}[Bilinear and Linear Forms]
For $\phi:\Omega\times [0,\infty)\to\mathbb{R}^+$, we define the bilinear functional $B_\phi(\cdot,\cdot): L^2([0,T];L^2(\Omega))\to L^2([0,T];L^2(\Omega))$ and linear functional $l:L^2([0,T];L^2(\Omega))\to L^2([0,T];L^2(\Omega))$ as
\begin{align}
B_\phi(\vec{v},\vec{u}) &:= \int_0^T\mathrm{d}t\,  \int  \mathrm{d}\vec{r}\,\vec{v}(\vec{r},t)\cdot \mathcal{H}_{\phi}\vec{u}(\vec{r},t),\\
l_n(\vec{v})&:= \int_0^T\mathrm{d}t\,  \int  \mathrm{d}\vec{r}\, \vec{v}(\vec{r},t)\cdot\vec{f}(\varrho_n,\varrho_{n-1}).
\end{align}
\end{definition}

With the weak formulations \eqref{eq:weak_formulation_pair_frozen_rho}--\eqref{eq:weak_formulation_pair_frozen_a} we establish the existence and uniqueness of the solution pair $\left(\varrho_n,\vec{a}_n\right)$ for each $n\in\mathbb{N}$.

\begin{lemma}[Existence and Uniqueness of weak $\left(\varrho_n(\vec{r},t),\, \vec{a}_n(\vec{r},t)\right)$]\label{lem:lax_milgram_n}
\end{lemma}
Let $\phi:\Omega\times[0,\infty)\to \mathbb{R}^{+}$ be a given probability density such that $\phi\in L^1([0,T];L^1(\Omega))$ and assume $\varrho_n\in C^\infty([0,T],C(\Omega))$ is the unique solution to \eqref{eq:weak_formulation_pair_frozen_rho}. Then, for every $n\in\mathbb{N}$, there exits a unique $\vec{a}_{n}\in L^2([0,T];L^2(\Omega,\phi^{-1}))$ such that \eqref{eq:weak_formulation_pair_frozen_a} holds for every $\vec{v}\in L^2([0,T]; L^2(\Omega,\phi^{-1}))$.
\begin{proof}
For equation \eqref{eq:weak_formulation_pair_frozen_a} and for each $n\in\mathbb{N}$ we establish existence and uniqueness of weak solutions
by using the Lax-Milgram Theorem \cite{lax1954ix}. We first check coercivity. Let $\phi:\Omega\times [0,\infty)\to \mathbb{R}^{+}$ and $\vec{u}(\vec{r},t)\in L^2([0,T];L^2(\Omega))$. By direct calculation we observe
\begin{align}
B_\phi(\vec{u},\vec{u}) &\geq \int_0^T\mathrm{d}t\, \int\mathrm{d}\vec{r}\, \phi^{-1}\vec{u}(\vec{r},t)\cdot \mathcal{H}_\phi \vec{u}(\vec{r},t)\\
 &= \int_0^T\mathrm{d}t\, \int\mathrm{d}\vec{r}\, \phi^{-1}\vec{u}(\vec{r},t)\bm{D}_\phi^{-1}\vec{u}(\vec{r},t) + \int_0^T\mathrm{d}t\, \int\mathrm{d}\vec{r}\, \phi^{-1}\vec{u}(\vec{r},t)\phi \bm{Z}_2\star \vec{u}\\
& \geq \|\mu_{max}\|_{L^\infty([0,T])}^{-1}\|\vec{u}\|_{L^2([0,T];L^2(\Omega,\phi^{-1}))} - \|\bm{Z}_2\|_{L^\infty}\int_0^T\mathrm{d}t\, \left(\int\mathrm{d}\vec{r}\,|\vec{u}(\vec{r},t)|\right)^2\\
&  \geq \|\mu_{max}\|_{L^\infty([0,T])}^{-1}\|\vec{u}\|_{L^2([0,T];L^2(\Omega,\phi^{-1}))} - \|\bm{Z}_2\|_{L^\infty}\int_0^T\mathrm{d}t\, \int\mathrm{d}\vec{r}\,\phi^{-1}|\vec{u}(\vec{r},t)|^2\\
& = \|\mu_{max}\|_{L^\infty([0,T])}^{-1}\left(1-\|\bm{Z}_2\|_{L^\infty}\|\mu_{max}\|_{L^\infty([0,T])}\right)  \|\vec{u}\|_{L^2([0,T];L^2(\Omega,\phi^{-1}))}\label{eq:H_rho_PD}
\end{align}
where in the last line we have used Jensen's inequality. From the  hypothesis \\ $1-\|\bm{Z}_2\|_{L^\infty}\|\mu_{\max}\|_{L^\infty([0,T])}>0$ and therefore $B_\phi(\cdot,\cdot)$ is coercive in $L^2([0,T];L^2(\Omega,\phi^{-1}))$. 

Now we show boundedness. We let $\vec{u},\vec{v}\in L^2([0,T];L^2(\Omega,\phi^{-1}))$ and see that
\begin{align}
&|B_\phi(\vec{v},\vec{u})| = \Big| \int_0^T\mathrm{d}t\, \int\mathrm{d}\vec{r}\, \phi^{-1}\vec{v}(\vec{r},t)\cdot \mathcal{H}_\phi \vec{u}(\vec{r},t) \Big| \\
& \leq  \int_0^T\mathrm{d}t\, \int\mathrm{d}\vec{r}\,\Big| \phi^{-1}\vec{v}(\vec{r},t)\cdot \mathcal{H}_\phi \vec{u}(\vec{r},t)\Big| \leq  \int_0^T\mathrm{d}t\, \int\mathrm{d}\vec{r}\,\phi^{-1}|\vec{v}(\vec{r},t)|\Big|\left(\bm{1}+\bm{Z}_1\star\phi\right)\vec{u}(\vec{r},t)+\phi\bm{Z}_2\star\vec{u}\Big|\\
& \leq  \|\mu_{\max}\|_{L^\infty([0,T])} \int_0^T\mathrm{d}t\, \int\mathrm{d}\vec{r}\,\phi^{-1/2} \phi^{-1/2} |\vec{v}(\vec{r},t)| |\vec{u}(\vec{r},t)| + \|\bm{Z}_2\|_{L^\infty}\int_0^T\mathrm{d}t\, \|\vec{v}\|_{L^1(\Omega)}\|\vec{u}\|_{L^1(\Omega)}\\
& \leq  \|\mu_{\max}\|_{L^\infty([0,T])} \int_0^T\mathrm{d}t\, \|\vec{v}\|_{L^2(\Omega,\phi^{-1})}\|\vec{u}\|_{L^2(\Omega,\phi^{-1})} + \|\bm{Z}_2\|_{L^\infty}\int_0^T\mathrm{d}t\, \|\phi\|_{L^1(\Omega)} \|\vec{v}\|_{L^2(\Omega,\phi^{-1})}\|\vec{u}\|_{L^2(\Omega,\phi^{-1})}\\
& \leq  \|\mu_{\max}\|_{L^\infty([0,T])} \int_0^T\mathrm{d}t\, \|\vec{v}\|_{L^2(\Omega,\phi^{-1})}\|\vec{u}\|_{L^2(\Omega,\phi^{-1})}\\
&\quad  + \|\bm{Z}_2\|_{L^\infty}\|\phi\|_{L^\infty([0,T];L^1(\Omega))}\int_0^T\mathrm{d}t\,  \|\vec{v}\|_{L^2(\Omega,\phi^{-1})}\|\vec{u}\|_{L^2(\Omega,\phi^{-1})}\\
& \leq  \left( \|\mu_{\max}\|_{L^\infty([0,T])} + \|\bm{Z}_2\|_{L^\infty}\|\phi\|_{L^\infty([0,T];L^1(\Omega))}\right) \|\vec{v}\|_{L^2([0,T];L^2(\Omega,\phi^{-1}))}\|\vec{u}\|_{L^2([0,T];L^2(\Omega,\phi^{-1}))},
\end{align}
where in the last line we have used H{\"o}lder's inequality. Therefore we find that $B_\phi(\cdot,\cdot)$ is bounded in $L^2([0,T];L^2(\Omega,\phi^{-1}))$. Finally we show boundedness of the right hand side. By direct calculation we see that
\begin{align}
&| l_n(\phi^{-1}\vec{v})| = \Big|\int_0^T\mathrm{d}t\,  \int  \mathrm{d}\vec{r}\, \vec{v}(\vec{r},t)\cdot\vec{f}(\varrho_n,\varrho_{n-1})\Big| \leq \int_0^T\mathrm{d}t\,  \int  \mathrm{d}\vec{r}\, \Big|\vec{v}(\vec{r},t)\cdot\vec{f}(\varrho_n,\varrho_{n-1})\Big|\\
& = \int_0^T\mathrm{d}t\, \int  \mathrm{d}\vec{r}\, \phi^{-1/2}\phi^{1/2}\Big|\vec{v}(\vec{r},t)\cdot\vec{f}(\varrho_n,\varrho_{n-1})\Big| \leq  \int_0^T\mathrm{d}t\, \int  \mathrm{d}\vec{r}\, \|\vec{v}\|_{L^2(\Omega,\phi^{-1})}\|\phi^{1/2}\vec{f}(\varrho_n,\varrho_{n-1})\|_{L^2(\Omega)}\\
& \leq  \|\vec{v}\|_{L^2([0,T];L^2(\Omega,\phi^{-1}))}\|\phi^{1/2}\vec{f}(\varrho_n,\varrho_{n-1})\|_{L^2([0,T];L^2(\Omega))}.
\end{align}
Thus by the Lax-Milgram Theorem \ref{thm:lax-milgram} there exists a unique $\vec{a}_n(\vec{r},t)\in {L^2([0,T];L^2(\Omega,\phi^{-1}))}$ solving \eqref{eq:weak_formulation_pair_frozen_a} for each $n\in \mathbb{N}$.
\end{proof}

We now prove the existence and uniqueness of a weak density $\varrho(\vec{r},t)$ solving \eqref{eq:weak_formulation_pair_rho}.

\begin{lemma}[Existence and Uniqueness of Weak Density]\label{thm:existence_and_uniqueness_weak_rho}

Let $\varrho_0 \in C^\infty(\Omega)$, $\varrho\geq 0$ and $\int  \mathrm{d}\vec{r} \varrho_0(\vec{r})=1$. Then there exists a unique weak solution $\varrho\in L^\infty([0,T];L^2(\Omega))\cap L^2([0,T]; H^1(\Omega))$, with
$\partial_t\varrho\in L^2([0,T]; H^{-1}(\Omega))$, to equation \eqref{eq:ddft_eq_dyn_rho} in the sense \eqref{eq:weak_formulation_pair_rho} with the estimate \eqref{eq:total_energy_bound}.
\end{lemma}

\begin{proof}

We establish sufficient boundedness of the sequence $\{\varrho_n\}{n\geq 1}$ by Propositions \ref{prop:bound_rhon}, \ref{prop:un_H1_rho0_bound}, and observe that by Lemma \ref{lem:rhon_is_cauchy} $\{\varrho_n\}{\geq 1}$ is Cauchy $L^1([0,T];L^1(\Omega))$ hence the existence of a weakly convergent subsequence $\{\varrho_{n_k}\}{k\geq 1}$ is guaranteed by the Banach--Alaoglu Theorem. The remainder of the argument is to show $\lim_{n\to\infty} \varrho_n$ exists and is a limit point solving the weak problem \eqref{eq:weak_formulation_pair_rho}. To ease notation we write $\vec{v}_{n_k}(\vec{r},t) :=  \nabla V_1(\vec{r},t) + \bm{A}(\vec{r},[\vec{a}_{n_k-1}],t)$. We multiply \eqref{eq:well_posed_un_ibvp} by $\eta\in L^2([0,T]; H^1(\Omega))$ after setting $n=n_k\in \mathbb{N}$ and integrate over $\Omega\times [0,T]$ to obtain
\begin{align}
\int_0^T\mathrm{d}t\, \langle \partial_t\varrho_{n_k},\, \eta(t) \rangle + \int_0^T\mathrm{d}t\,\int \mathrm{d}\vec{r}\, \nabla \eta \cdot\bm{D}\,(\nabla \varrho_{n_k}+	\varrho_{n_k} (\vec{v}_{n_k}+ \,\nabla V_2\star \varrho_{n_k-1}))=0.
\end{align}
For the transport term we write
\begin{align}
&\int_0^T\mathrm{d}t\, \nabla \eta \cdot \varrho_{n_k}\bm{D}\, (\vec{v}_{n_k}+ \,\nabla V_2\star \varrho_{n_k-1})\nonumber\\
& = \int_0^T\mathrm{d}t\, \nabla \eta \cdot\,(\varrho_{n_k}-\varrho)\bm{D}\, (\vec{v}_{n_k}+ \,\nabla V_2\star \varrho_{n_k-1})\\
&\quad +\int_0^T\mathrm{d}t\, \nabla \eta \cdot \,\varrho \bm{D}\, (\vec{v}_{n_k}+ \,\nabla V_2\star (\varrho_{n_k-1}-\varrho)) +\int_0^T\mathrm{d}t\, \nabla \eta\cdot \,\varrho \bm{D}\,  \,\nabla V_2\star \varrho.
\end{align}
Note that $\varrho_{n_k}\rightharpoonup\varrho$ in $L^2([0,T];H^1(\Omega))\subset L^2([0,T];L^2(\Omega))$ and $(\nabla \cdot\bm{D})\cdot (\vec{v}_{n_k}(\vec{r},t)+ \,\nabla V_2\star (\varrho_{n_k-1}))$ is uniformly bounded and so
\begin{align}
\int_0^T\mathrm{d}t\, \int \mathrm{d}\vec{r}\,\nabla ^\top\eta \, (\varrho_{n_k}-\varrho)\bm{D}\, (\vec{v}_{n_k}+ \,\nabla V_2\star \varrho_{n_k-1})\to 0
\end{align}
as $k\to \infty$. 

Now by H{\"o}lder's inequality one has
\begin{align}
&\int_0^T\mathrm{d}t\, \nabla \eta\cdot \,\varrho\,\bm{D}\,\nabla  (V_2\star (\varrho_{n_k-1}-\varrho))\\
&\quad \leq \mu_{\max}\|\nabla \eta \|_{L^2([0,T]; L^2(\Omega))}\|\nabla V_2\|_{L^{\infty}(\Omega)} \left( \int_0^T\mathrm{d}t\,\|\varrho_{n_k-1}(t)-\varrho(t)\|_{L^1(\Omega)}^2\right)^{1/2}\to 0.
\end{align}
Now note that by Lemma \ref{lem:rhon_is_cauchy}, $\|\phi_n\|_{L^1(\Omega)}$ is bounded and therefore
\begin{align}
\int_0^T\mathrm{d}t \|\varrho_{n_k-1}(t)-\varrho(t)\|_{L^1(\Omega)}^2\leq C\int_0^T\mathrm{d}t \|\varrho_{n_k-1}(t)-\varrho(t)\|_{L^1([0,T];L^1(\Omega))}
\to 0.
\end{align}

Therefore we have
\begin{align}
\int_0^T\mathrm{d}t\, \nabla \eta\, \cdot \,\varrho_{n_k}\bm{D}\, (\vec{v}_{n_k}+ \,\nabla V_2\star \varrho_{n_k-1})
\to\int_0^T\mathrm{d}t\, \nabla \eta \cdot\varrho  \bm{D}\, (\vec{v}_{n_k}+ \,\nabla V_2\star \varrho)
\end{align}
as $k\to \infty$.
By the weak convergence results of Lemma \ref{lem:weak_conv_results} we have
\begin{align}
\int_0^T\mathrm{d}t\, \langle \partial_t \varrho_{n_k},\, \varrho_{n_k} \rangle &\to \int_0^T\mathrm{d}t\, \langle \partial_t\varrho,\, \varrho \rangle,\nonumber\\
\int_0^T\mathrm{d}t\,\int \mathrm{d}\vec{r}\, \nabla \eta \cdot\bm{D}\,\nabla \varrho_{n_k} &\to \int_0^T\mathrm{d}t\,\int \mathrm{d}\vec{r}\, \nabla \eta \cdot\bm{D}\,\nabla \varrho
\end{align}
as $k\to \infty$. This establishes existence of weak solution to \eqref{eq:ddft_dyn_non_gradient_flow} in the sense \eqref{eq:weak_formulation_pair_rho}. Establishing $\varrho(0)=\varrho_0$ is a routine argument (see \cite{evans2002partial}). 

To prove uniqueness we set $\xi = \varrho_1-\varrho_2$ where $\varrho_1,\varrho_2$ are weak solutions then we have
\begin{align}
\int_0^T\mathrm{d}t\, \langle \partial_t \xi(t), \, \eta(t)  \rangle+\int_0^T\mathrm{d}t\, \int  \mathrm{d}\vec{r}\, \nabla \eta\cdot \bm{D}\,(\nabla \xi
 +\xi\,\vec{v}_{n_k}+ \varrho_1\nabla \,V_2\star\varrho_1- \varrho_1\nabla \,V_2\star\varrho_2)=0
\end{align}
Adding and subtracting $\int_0^T\mathrm{d}t\,\int \mathrm{d}\vec{r}'\,\nabla \eta\cdot  \varrho_2\nabla V_2\star \varrho_1$ we find
\begin{align}
&\int_0^T\mathrm{d}t\, \langle \partial_t \xi(t), \, \eta(t)  \rangle + \int_0^T\mathrm{d}t\, \int  \mathrm{d}\vec{r}\, \nabla \eta\cdot \bm{D}\,\nabla \xi =-\int_0^T\mathrm{d}t\, \int  \mathrm{d}\vec{r}\, \nabla \eta\cdot \bm{D}\,(\xi\,\vec{v}_{n_k}+ \xi\nabla \,V_2\star\varrho_1- \varrho_2\nabla \,V_2\star\xi)\nonumber\\
&\quad\leq \int_0^T\mathrm{d}t\, \int  \mathrm{d}\vec{r}\, |\nabla \eta\cdot \bm{D}^{1/2} \bm{D}^{1/2}\,(\xi\,\vec{v}_{n_k}+ \xi\nabla \,V_2\star\varrho_1- \varrho_2\nabla \,V_2\star\xi)|.\label{eq:uniqueness_bound_1}
\end{align}
By Young's inequality we have
\begin{align}
&\int_0^T\mathrm{d}t\, \int  \mathrm{d}\vec{r}\, |\nabla \eta\cdot \bm{D}^{1/2} \bm{D}^{1/2}\,(\xi\,\vec{v}_{n_k}+ \xi\nabla \,V_2\star\varrho_1- \varrho_2\nabla \,V_2\star\xi)|\nonumber\\
&\quad\leq \int_0^T\mathrm{d}t\, \int  \mathrm{d}\vec{r}\, |\bm{D}^{1/2} \nabla \eta |^2 +\tfrac{1}{4}\int_0^T\mathrm{d}t\, \int  \mathrm{d}\vec{r}\, | \bm{D}^{1/2} (\xi\,\vec{v}_{n_k}+ \xi\nabla \,V_2\star\varrho_1- \varrho_2\nabla \,V_2\star\xi ) |^2.
\end{align}
Using the triangle inequality and Young's inequality we expand the absolute value inside the integral
\begin{align}
&\tfrac{1}{4}\int_0^T\mathrm{d}t\, \int  \mathrm{d}\vec{r}\, |\bm{D}^{1/2} ( \xi\,\vec{v}_{n_k}+ \xi\nabla \,V_2\star\varrho_1- \varrho_2\nabla \,V_2\star\xi )|^2\nonumber\\
&\quad \leq \tfrac{1}{4}\int_0^T\mathrm{d}t\, \int  \mathrm{d}\vec{r}\, |\bm{D}^{1/2} \xi\,\vec{v}_{n_k}|^2
+ |\bm{D}^{1/2} ( \xi\nabla \,V_2\star\varrho_1-\varrho_2\nabla \,V_2\star\xi ) |^2 \nonumber\\
&\quad\leq \tfrac{1}{4}\int_0^T\mathrm{d}t\, \int  \mathrm{d}\vec{r} \left(  |\bm{D}^{1/2}\xi\,\vec{v}_{n_k}|^2
+2 |\bm{D}^{1/2}\xi\nabla \,V_2\star\varrho_1|^2
+2 |\bm{D}^{1/2}\varrho_2\nabla \,V_2\star\xi |^2\right)\nonumber\\
&\quad\leq \tfrac{\mu_{\max}}{4}\int_0^T\mathrm{d}t\, \int  \mathrm{d}\vec{r}\,  |\xi\,\vec{v}_{n_k}|^2+2 |\xi\nabla \,V_2\star\varrho_1|^2+2 |\varrho_2\nabla \,V_2\star\xi |^2. \label{eq:uniqueness_bound_2}
\end{align}
Estimating each of these terms, first
\begin{align}\label{eq:uniqueness_bound_3}
\int_0^T\mathrm{d}t\, \int  \mathrm{d}\vec{r}\,  |\xi\,\vec{v}_{n_k}|^2\leq\|\vec{v}_{n_k}\|_{L^\infty(\Omega)}^2\|\xi\|_{L^2([0,T];L^2(\Omega))}.
\end{align}
Second,
\begin{align}\label{eq:uniqueness_bound_4}
2 \int_0^T\mathrm{d}t\, \int  \mathrm{d}\vec{r}\,  |\xi\nabla \,V_2\star\varrho_1|^2 \leq 2 |\Omega\||\nabla V_2\|_{L^\infty(\Omega)}^2\|\xi\|_{L^2([0,T];L^2(\Omega))},
\end{align}
and third
\begin{align}
2 \int_0^T\mathrm{d}t\, \int  \mathrm{d}\vec{r}\,|\varrho_2\nabla \,V_2\star\xi|^2\leq 2 |\Omega| \|\varrho_2\|_{L^\infty([0,T];L^2(\Omega))} \|\nabla V_2\|_{L^\infty(\Omega)}^2\|\xi\|_{L^2([0,T];L^2(\Omega))}.\label{eq:uniqueness_bound_5}
\end{align}

Combining \eqref{eq:uniqueness_bound_1}, \eqref{eq:uniqueness_bound_2}, \eqref{eq:uniqueness_bound_3}, \eqref{eq:uniqueness_bound_4}, \eqref{eq:uniqueness_bound_5} we obtain, after setting $\eta = \xi$, and using boundedness of $\varrho_2$ in terms of its initial data
\begin{align}
\int_0^T\mathrm{d}t\, \langle \partial_t \xi(t),\, \xi(t)\rangle  \leq (C_1(T)+C_2(T)\|\varrho_0\|_{L^2(\Omega)}^2)\|\xi \|_{L^2([0,T];L^2(\Omega))}^2
\end{align}
for some constants $C_1(T)$, $C_2(T)$ dependent on $\Omega$. This holds for all $T$ so it must be the case that
\begin{align}
\der[]{t}\|\xi(t)\|_{L^2(\Omega)}^2\leq (C_1(T)+C_2(T)\|\varrho_0\|_{L^2(\Omega)}^2)\|\xi(t)\|_{L^2(\Omega)}^2
\end{align}
implying by Gr{\"o}nwall's lemma that
\begin{align}
\|\xi(t)\|_{L^2(\Omega)} \leq (C_1(T)+C_2(T)\|\varrho_0\|_{L^2(\Omega)}^2)\|\xi(0)\|_{L^2(\Omega)}
\end{align}
a.e.\ $t\in [0,T]$. However, $\xi(0)\equiv 0$ hence $\|\varrho_1(t)-\varrho_2(t)\|_L^2(\Omega)=0$ for all $t\in [0,T]$.
\end{proof}

We now prove the existence and uniqueness of a weak flux $\vec{a}(\vec{r},t)$ solving \eqref{eq:weak_formulation_pair_a}.

\begin{lemma}[Existence and Uniqueness of Weak Flux]\label{thm:existence_and_uniqueness_weak_a}
Let $\varrho_0 \in C^\infty(\Omega)$, $\varrho\geq 0$ and $\int  \mathrm{d}\vec{r} \varrho_0(\vec{r})=1$. Then there exists a unique weak solution $\vec{a}\in L^\infty([0,T];L^2(\Omega))$ to equation \eqref{eq:ddft_eq_dyn_a} in the sense \eqref{eq:weak_formulation_pair_a}.
\end{lemma}
\begin{proof}
By Lemma \ref{lem:lax_milgram_n} for each $n\in\mathbb{N}$ there exists a unique $\vec{a}_n(\vec{r},t)\in L^2([0,T];L^2(\Omega,\phi^{-1}))$. We now consider the following $L^1([0,T];L^1(\Omega))$ norm
\begin{align}
\|\vec{a}_n\|_{L^1([0,T];L^1(\Omega))} &= \int_0^T\mathrm{d}t\, \int\mathrm{d}\vec{r}\, |\vec{a}_n(\vec{r},t)| = \int_0^T\mathrm{d}t\, \int\mathrm{d}\vec{r}\, \phi^{-1/2}\phi^{1/2}|\vec{a}_n(\vec{r},t)|\\
& \leq \|\phi\|_{L^1([0,T];L^1(\Omega))}\|\vec{a}_n\|_{L^2([0,T];L^2(\Omega,\phi^{-1}))}.
\end{align}
We now let $\phi = \varrho_{n-2}$ and note that by Lemma \ref{lem:rhon_is_cauchy} $\varrho_{n-2}\in L^1([0,T];L^1(\Omega))$. Additionally by Lemma \ref{lem:lax_milgram_n} $\vec{a}_n\in L^2([0,T];L^2(\Omega,\phi^{-1}))$ hence we see that $\vec{a}_n\in L^1([0,T];L^1(\Omega))$ for every $n\in\mathbb{N}$. Now since $L^1([0,T];L^1(\Omega))$ is reflexive, by Banach--Alaoglu there exists a weakly converging subsequence $\{\vec{a}_{n_k}\}_{k\geq 1}$ such that $\vec{a}_{n_k}\rightharpoonup \vec{a}$ (weakly in $L^1([0,T];L^1(\Omega))$) as $k\to \infty$ to some $\vec{a}\in L^1([0,T];L^1(\Omega))$. Now consider the left hand side of \eqref{eq:weak_formulation_pair_frozen_a}. Let $\vec{v}(\vec{r},t)\in L^2([0,T];L^2(\Omega))$ we have that 
\begin{align}
\int_0^T\mathrm{d}t\,\langle \vec{v}, \mathcal{H}_{\phi}\vec{a}_n\rangle_{L^2(\Omega))} &= \int_0^T\mathrm{d}t\,\langle \varrho_{n-2}\vec{v}, \mathcal{H}_{\varrho_{n-2}}\vec{a}_n\rangle_{L^2(\Omega,\varrho_{n-2}^{-1}))}\\
& =  \int_0^T\mathrm{d}t\,\langle \mathcal{H}_{\varrho_{n-2}}^{\ast}\varrho_{n-2}\vec{v}, \vec{a}_n\rangle_{L^2(\Omega,\varrho_{n-2}^{-1}))}\\
&=\int_0^T\mathrm{d}t\,\langle \mathcal{H}_{\varrho_{n-2}}\varrho_{n-2}\vec{v}, \vec{a}_n\rangle_{L^2(\Omega))}\\
& = \int_0^T\mathrm{d}t\,\int\mathrm{d}\vec{r}\,\left(\bm{D}_{\varrho_{n-2}}\vec{v}(\vec{r},t) + \bm{Z}_2\star\left(\varrho_{n-2}\vec{v}\right)\right)\vec{a}_n(\vec{r},t),
\end{align}
where we have used the self adjoint property of $\mathcal{H}_\phi$. All that remains is to show that $\left(\bm{D}_{\varrho_{n-2}}\vec{v}(\vec{r},t) + \bm{Z}_2\star\left(\varrho_{n-2}\vec{v}\right)\right)\in L^1([0,T];L^1(\Omega))$. In particular we have that
\begin{align}
&\| \bm{D}_{\varrho_{n-2}}\vec{v}(\vec{r},t) + \bm{Z}_2\star\left(\varrho_{n-2}\vec{v}\right)\|_{L^1([0,T];L^1(\Omega))} =  \int_0^T\mathrm{d}t\,\int\mathrm{d}\vec{r}\,\big| \bm{D}_{\varrho_{n-2}}\vec{v}(\vec{r},t) + \bm{Z}_2\star\left(\varrho_{n-2}\vec{v}\right)\big|\\
&\leq \int_0^T\mathrm{d}t\,\int\mathrm{d}\vec{r}\,\big| \bm{D}_{\varrho_{n-2}}\vec{v}(\vec{r},t) \big|+ \big|\bm{Z}_2\star\left(\varrho_{n-2}\vec{v}\right)\big|\\
&\leq \|\bm{D}_{n-2}\|_{L^2([0,T];L^2(\Omega))}\|\vec{v}\|_{L^2([0,T];L^2(\Omega))}\\
&\quad + |\Omega|\|\bm{Z}_2\|_{L^\infty}\|\varrho_{n-2}\|_{L^2([0,T];L^2(\Omega))}\|\vec{v}\|_{{L^2([0,T];L^2(\Omega))}}<\infty.
\end{align}
We now set $n_{k} = n$ and pass to the limit $k\to\infty$
\begin{align}
&\lim_{k\to\infty}\int_0^T\mathrm{d}t\,\langle \vec{v}, \mathcal{H}_{\varrho_{n_k-2}}\vec{a}_{n_k}\rangle_{L^2(\Omega))}\\
\quad &= \lim_{k\to\infty}\int_0^T\mathrm{d}t\,\int\mathrm{d}\vec{r}\,\left(\bm{D}_{\varrho_{n_k-2}}\vec{v}(\vec{r},t) + \bm{Z}_2\star\left(\varrho_{n_k-2}\vec{v}\right)\right)\vec{a}_{n_k}(\vec{r},t)\\
& = \int_0^T\mathrm{d}t\,\int\mathrm{d}\vec{r}\,\left(\bm{D}_{\varrho}\vec{v}(\vec{r},t) + \bm{Z}_2\star\left(\varrho\vec{v}\right)\right)\vec{a}(\vec{r},t) = \int_0^T\mathrm{d}t\,\langle \vec{v}, \mathcal{H}_{\varrho}\vec{a}\rangle_{L^2(\Omega))}
\end{align}
where we have used the weak convergence of $\vec{a}_{n_k}$ and weak convergence of $\varrho_{n_k}$ (see Corollary \ref{cor:summary_of_convergence}).  Now consider the right hand side of \eqref{eq:weak_formulation_pair_frozen_a}
\begin{align}
\lim_{k\to\infty}\int_0^T\mathrm{d}t\,  \int  \mathrm{d}\vec{r}\, \vec{v}(\vec{r},t)\cdot\vec{f}(\varrho_{n_k},\varrho_{n_k-1}) = \int_0^T\mathrm{d}t\,  \int  \mathrm{d}\vec{r}\, \vec{v}(\vec{r},t)\cdot\vec{f}(\vec{r},\varrho,t),
\end{align}
where we have used the weak convergence of $\varrho_{n_k}$ from Corollary \ref{cor:summary_of_convergence}.

To obtain uniqueness we suppose $\vec{a}(\vec{r},t)$ and $\vec{b}(\vec{r},t)\in L^2([0,T];L^2(\Omega))$ are both solutions to \eqref{eq:weak_formulation_pair_a} then, for every $\vec{v}(\vec{r},t)\in L^2([0,T];L^2(\Omega))$ we have that
\begin{align}
0 = \int_0^T\mathrm{d}t\,\langle \vec{v}, \mathcal{H}_{\varrho}(\vec{a}-\vec{b})\rangle_{L^2(\Omega))} = \int_0^T\mathrm{d}t\,\langle \varrho\vec{v}, \mathcal{H}_{\varrho}(\vec{a}-\vec{b})\rangle_{L^2(\Omega,\varrho^{-1}))}.
\end{align}
We now let $\vec{v} = (\vec{a}-\vec{b})/\varrho(\vec{r},t)$, where $\varrho(\vec{r},t)$ is the unique weak solution to \eqref{eq:weak_formulation_pair_rho}, to obtain 
\begin{align}
0 = \int_0^T\mathrm{d}t\,\langle (\vec{a}-\vec{b}), \mathcal{H}_{\varrho}(\vec{a}-\vec{b})\rangle_{L^2(\Omega,\varrho^{-1}))} &= \int_0^T\mathrm{d}t\,\langle (\vec{a}-\vec{b}), \mathcal{H}^{1/2}_{\varrho}\mathcal{H}^{1/2}_{\varrho}(\vec{a}-\vec{b})\rangle_{L^2(\Omega,\varrho^{-1}))} \\
&= \int_0^T\mathrm{d}t\,\|\mathcal{H}^{1/2}_{\varrho}(\vec{a}-\vec{b})\|_{L^2(\Omega,\varrho^{-1}))}^2,
\end{align}
and, by using the positive definiteness of $\mathcal{H}_{\varrho}$ in $ L^2([0,T];L^2(\Omega,\varrho^{-1}))$ (by equation \eqref{eq:H_rho_PD}), we obtain that $\vec{a}(\vec{r},t) = \vec{a}(\vec{r},t)$ a.e. $\vec{r}\in \Omega$ and $t\in [0,T]$. 
\end{proof}

By combining Lemma \ref{thm:existence_and_uniqueness_weak_rho} and Lemma \ref{thm:existence_and_uniqueness_weak_a} at once we obtain Theorem \ref{thm:exis_uniq_rho_and_a}. We conclude this section with a variational principle for the time dependent flux $\vec{a}(\vec{r},t)$. It turns out that \eqref{eq:weak_formulation_pair_a} can be restated as a minimisation problem. To be more precise, we define the quadratic functional $\mathcal{J}:L^2([0,T];L^2(\Omega))\to \mathbb{R}$ by
\begin{align}
\mathcal{J}[\vec{v}] := \int_0^T\mathrm{d}t\,  \int  \mathrm{d}\vec{r}\,\frac{1}{2}\varrho(\vec{r},t)^{-1}\vec{v}(\vec{r},t)\cdot\left(\mathcal{H}_{\varrho}\vec{v}\right)(\vec{r},t)  -\vec{v}(\vec{r},t)\cdot\nabla \frac{\delta\mathcal{F}}{\delta\varrho}[\varrho](\vec{r},t),
\end{align}
for $\vec{v}(\vec{r},t) \in :L^2([0,T];L^2(\Omega))$. For the following proposition we define the following bilinear and linear functionals
\begin{align}
B(\vec{v},\vec{u}) &:= \int_0^T\mathrm{d}t\,  \int  \mathrm{d}\vec{r}\,\vec{v}(\vec{r},t)\cdot \mathcal{H}_{\phi}\vec{u}(\vec{r},t),\\
l(\vec{v})&:= \int_0^T\mathrm{d}t\,  \int  \mathrm{d}\vec{r}\, \vec{v}(\vec{r},t)\cdot\varrho\nabla \frac{\delta\mathcal{F}}{\delta\varrho}[\varrho].
\end{align}

\begin{proposition}[$\vec{a}(\vec{r},t)$ is weak solution iff $\vec{a}(\vec{r},t)$ minimises $\mathcal{J}$.]
Let $B(\cdot,\cdot)$ be the bilinear functional on $L^2([0,T];L^2(\Omega))$ in \eqref{eq:weak_formulation_pair_a}. Then, $\vec{a}(\vec{r},t)\in L^2([0,T];L^2(\Omega))$ is the (unique) weak solution to \eqref{eq:weak_formulation_pair_a} if and only if $\vec{a}(\vec{r},t)$ is the unique minimiser of $\mathcal{J}[\cdot]$ over $L^2([0,T];L^2(\Omega))$.
\end{proposition}
\begin{proof}
We let $\vec{a}(\vec{r},t)\in L^2([0,T];L^2(\Omega))$ be the unique weak solution to \eqref{eq:weak_formulation_pair_a} and for every $\vec{v}(\vec{r},t)\in L^2([0,T];L^2(\Omega))$, consider $\mathcal{J}[\vec{v}]-\mathcal{J}[\vec{a}]$, 
\begin{align}
\mathcal{J}[\vec{v}]-\mathcal{J}[\vec{a}] &= \frac{1}{2}B(\varrho^{-1}\vec{v},\vec{v}) - l(\varrho^{-1}\vec{v})- \frac{1}{2}B(\varrho^{-1}\vec{a},\vec{a}) + l(\varrho^{-1}\vec{a})\\
& = \frac{1}{2}B(\varrho^{-1}\vec{v},\vec{v}) - \frac{1}{2}B(\varrho^{-1}\vec{a},\vec{a}) - l(\varrho^{-1}(\vec{v}-\vec{a}))\\
& =  \frac{1}{2}B(\varrho^{-1}\vec{v},\vec{v}) - \frac{1}{2}B(\varrho^{-1}\vec{a},\vec{a}) - B(\varrho^{-1}\vec{a},\vec{v}-\vec{a}))\\
& = \frac{1}{2}\left(B(\varrho^{-1}\vec{v},\vec{v}) -2 B(\varrho^{-1}\vec{a},\vec{v}) + B(\varrho^{-1}\vec{a},\vec{a})\right)\\
& = \frac{1}{2}\left(B(\varrho^{-1}\vec{v},\vec{v})-B(\varrho^{-1}\vec{a},\vec{v})-B(\varrho^{-1}\vec{v},\vec{a})+B(\varrho^{-1}\vec{a},\vec{a})\right)\\
& = \frac{1}{2}B(\varrho^{-1}(\vec{v}-\vec{a}),\vec{v}-\vec{a}),
\end{align}
where we have successively used the fact that $B(\cdot,\cdot) $ is symmetric in $L^2([0,T];L^2(\Omega,\varrho^{-1}))$. Therefore, we find $\mathcal{J}[\vec{v}]-\mathcal{J}[\vec{a}]  = 1/2B(\varrho^{-1}(\vec{v}-\vec{a}),\vec{v}-\vec{a})$, and, using the positive definiteness of $\mathcal{H}_\varrho$ in $L^2([0,T];L^2(\Omega,\varrho^{-1}))$, find 
\begin{align}
\mathcal{J}[\vec{v}]-\mathcal{J}[\vec{a}] \geq \delta \|\vec{v}-\vec{a}\|_{L^2([0,T];L^2(\Omega,\varrho^{-1}))}\label{eq:J_difference_bounded_below}
\end{align}
for some $\delta>0$ and hence
\begin{align}
\mathcal{J}[\vec{v}] \geq J[\vec{a}],\label{eq:a_minimises_J}
\end{align} 
i.e., $\vec{a}$ minimises $\mathcal{J}[\cdot]$ in $L^2([0,T];L^2(\Omega))$. In fact $\vec{a}$ is the unique minimiser of\\ $\mathcal{J}$ in $L^2([0,T];L^2(\Omega))$. In particular if $\vec{b}(\vec{r},t)\in L^2([0,T];L^2(\Omega))$ also minimises $\mathcal{J}$ then $J[\vec{v}] \geq J[\vec{b}]$ for every $\vec{v}(\vec{r},t)\in L^2([0,T];L^2(\Omega))$. Taking $\vec{v} = \vec{b}$ in \eqref{eq:a_minimises_J} we find that $\mathcal{J}[\vec{b}] =J[\vec{a}]$, then by virtue of \eqref{eq:J_difference_bounded_below} we obtain $\|\vec{b}-\vec{a}\|_{L^2([0,T];L^2(\Omega,\varrho^{-1}))} = 0$ and hence $\vec{b}(\vec{r},t)=\vec{a}(\vec{r},t)$ a.e. $\vec{r}\in \Omega$ and $t\in [0,T]$. 

We now suppose that $\vec{a}(\vec{r},t)\in L^2([0,T];L^2(\Omega))$ is a minimiser of $\mathcal{J}[\cdot]$. By direct calculation we see that $\mathcal{J}$ is convex. In particular let $\vec{v}(\vec{r},t),\vec{w}\in (\vec{r},t)\in L^2([0,T];L^2(\Omega))$ and let $\theta\in [0,1]$ then
\begin{align}
\mathcal{J}[(1-\theta)\vec{v}+\theta\vec{w}] = (1-\theta)\mathcal{J}[\vec{v}]+ \theta\mathcal{J}[\vec{w}] +\frac{1}{2}B(\varrho^{-1}(\vec{v}-\vec{w}),\vec{v}-\vec{w}).
\end{align}
Then using the positive definiteness property of $\mathcal{H}_\varrho$ in $ L^2([0,T];L^2(\Omega,\varrho^{-1}))$ we have that $B(\varrho^{-1}(\vec{v}-\vec{w}),\vec{v}-\vec{w})\geq 0$ and conclude that  $\mathcal{J}[(1-\theta)\vec{v}+\theta\vec{w}] \leq (1-\theta)\mathcal{J}[\vec{v}]+ \theta\mathcal{J}[\vec{w}]$. Moreover, if $\vec{a}$ minimises $\mathcal{J}[\cdot]$ then $\mathcal{J}[\cdot]$ has a stationary point at $\vec{a}$, such that
\begin{align}
\lim_{\epsilon\to 0}\frac{\mathcal{J}[\vec{a}+\epsilon\vec{v}]-\mathcal{J}[\vec{v}]}{\epsilon} = 0,
\end{align}
for every $\vec{v}\in L^2([0,T];L^2(\Omega))$. But since 
\begin{align}
\frac{\mathcal{J}[\vec{a}+\epsilon\vec{v}]-\mathcal{J}[\vec{v}]}{\epsilon} = B(\varrho^{-1}\vec{a},\vec{v})-l(\varrho^{-1}\vec{v}) +\frac{\epsilon}{2}B(\varrho^{-1}\vec{a},\vec{a}),
\end{align}
we deduce that, after taking $\epsilon\to 0$, 
\begin{align}
B(\varrho^{-1}\vec{a},\vec{v})-l(\varrho^{-1}\vec{v}) = 0,
\end{align}
for every $\vec{v}(\vec{r},t)\in L^2([0,T];L^2(\Omega))$ and where $\varrho(\vec{r},t)\in L^2([0,T];L^2(\Omega))$ is the unique weak solution to \eqref{eq:weak_formulation_pair_rho}. We then see that
\begin{align}
0 &= B(\varrho^{-1}\vec{a},\vec{v})-l(\varrho^{-1}\vec{v})\\
&= \int_0^T\,\mathrm{d}t\langle \vec{a},\mathcal{H}_\varrho\vec{v} \rangle_{L^2(\Omega,\varrho^{-1})}-\langle\varrho^{-1}\vec{v}, \varrho\nabla \frac{\delta\mathcal{F}}{\delta\varrho}[\varrho]\rangle_{L^2(\Omega)}\\
& = \int_0^T\mathrm{d}t\,\langle \mathcal{H}_\varrho^\ast\vec{a},\vec{v} \rangle_{L^2(\Omega,\varrho^{-1})}-\langle\varrho^{-1}\vec{v}, \varrho\nabla \frac{\delta\mathcal{F}}{\delta\varrho}[\varrho]\rangle_{L^2(\Omega)}\\
& = \int_0^T\mathrm{d}t\,\langle \mathcal{H}_\varrho\vec{a},\vec{v} \rangle_{L^2(\Omega,\varrho^{-1})}-\langle\varrho^{-1}\vec{v}, \varrho\nabla \frac{\delta\mathcal{F}}{\delta\varrho}[\varrho]\rangle_{L^2(\Omega)}\\
& = \int_0^T\mathrm{d}t\,\langle \vec{v},\mathcal{H}_\varrho\vec{a} \rangle_{L^2(\Omega,\varrho^{-1})}-\langle\vec{v}, \varrho\nabla \frac{\delta\mathcal{F}}{\delta\varrho}[\varrho]\rangle_{L^2(\Omega,\varrho^{-1})}.\label{eq:a_minimiser_is_weak_sol_in_L2_rho_inv}
\end{align}
Note that $\vec{v}(\vec{r},t)\in L^2([0,T];L^2(\Omega))$ was arbitrary hence we check that $\varrho^{-1}(\vec{r},t)\vec{v}(\vec{r},t)\in L^2([0,T];L^2(\Omega))$, where $\varrho(\vec{r},t)$ is the unique weak solution to \eqref{eq:weak_formulation_pair_rho}. Indeed, we have that
\begin{align}
\|\varrho^{-1}\vec{v}\|_{L^2([0,T];L^2(\Omega))}^2 \leq \delta^2 \|\vec{v}\|_{L^2([0,T];L^2(\Omega))}^2,
\end{align}
for some delta such that $\varrho(\vec{r},t)>0$ for every $\vec{r}\in \Omega$ and $t\in [0,\infty)$. Hence we may write $\varrho^{-1}(\vec{r},t)\vec{v}(\vec{r},t) = \vec{w}(\vec{r},t)$ for some $\vec{w}(\vec{r},t)\in L^2([0,T];L^2(\Omega))$ and see that \eqref{eq:a_minimiser_is_weak_sol_in_L2_rho_inv} becomes
\begin{align}
0 = \int_0^T\mathrm{d}t\,\langle \vec{w},\mathcal{H}_\varrho\vec{a} \rangle_{L^2(\Omega)}-\langle\vec{w}, \varrho\nabla \frac{\delta\mathcal{F}}{\delta\varrho}[\varrho]\rangle_{L^2(\Omega)},
\end{align}
for every $\vec{w}(\vec{r},t)\in L^2([0,T];L^2(\Omega))$, and implies \eqref{eq:weak_formulation_pair_a}.
\end{proof}

We now consider the behaviour of classical solutions, that is, solutions\\ $\varrho(\vec{r},t)\in C^1([0,\infty);C^2(\Omega))$ and $\vec{a}(\vec{r},t)\in  C([0,\infty);C^{2d}(\Omega))$ such that \eqref{eq:ddft_eq_dyn_rho}--\eqref{eq:ddft_eq_dyn_a} holds pointwise.

\section{Behaviour of Classical Solutions}\label{sec:behaviour_classical_solutions}

We return to equations \eqref{eq:ddft_eq_dyn_rho}--\eqref{eq:ddft_eq_dyn_a}. We will determine that the contraction condition $\|\mu_{\max}\|_{L^\infty([0,T])}\|\bm{Z}_2\|_{L^{\infty}(\Omega)}<1$ is necessary for the invertibility of $\mathcal{H}_{\varrho}$ in equation \eqref{eq:ddft_eq_dyn_a}. Note that, for a fixed density field $\varrho(\vec{r},t)$, \eqref{eq:ddft_eq_dyn_a} may be written in terms of the generalised HI operator \eqref{def:H_phi_op}, as
\begin{align}\label{eq:matrix_operator_friction_eqn}
\mathcal{H}_{\varrho}\vec{a}(\vec{r},t)=\varrho(\vec{r},t)\nabla \frac{\delta\mathcal{F}}{\delta\varrho}[\varrho](\vec{r},t).
\end{align}
In Lemma \ref{lem:H_lambda_is_compact_SA} we establish that $\mathcal{H}_{\varrho}$ is a compact and self-adjoint operator in the weighted space $L^2([0,T];L^2(\Omega, \varrho^{-1}))$. By the Spectral Theorem we may therefore diagonalise $\mathcal{H}_{\varrho}$ by
\begin{align}
\mathcal{H}_{\varrho}\vec{u}_k(\vec{r},t) = \gamma_k\vec{u}(\vec{r},t)\label{eq:diagonalisation_of_H_rho}
\end{align}
where the eigenfunctions $\{\vec{u}_k\}_{k \in\mathbb{N}}$ forms a complete orthonormal basis of $L^2([0,T];L^2(\Omega, \varrho^{-1}))$ and the eigenvalues $\{\gamma_k\}_{k \in\mathbb{N}}\in\mathbb{R}$ for each $k\in\mathbb{N}$. Furthermore, by since $\mathcal{H}_{\varrho}$ is a real and symmetric (Hermitian) operator, $\mathcal{H}^{-1}_{\varrho}$ may be written as a Laplace transform by the Hille--Yosida theorem. Therefore, in principle, $\mathcal{H}^{-1}_{\varrho}$ can be obtained explicitly, at least in terms of a power series.
\\\\
\indent The following lemma establishes a solvability condition for the flux equation in equation \eqref{eq:matrix_operator_friction_eqn}.

\begin{proposition}[Conditional Invertibility of $\mathcal{H}_{\varrho}$]\label{thm:cond_converg_fred_det}
\par
Let $\|\mu_{\max}\|_{L^\infty([0,T])}\|\bm{Z}_2\|_{L^\infty(\Omega)}<1$ then, for the fixed density field $\varrho(\vec{r},t)$, the operator $\mathcal{H}_{\varrho}$ is invertible.
\end{proposition}

\begin{proof}
For each $\varrho:\Omega \times [0,\infty) \to \mathbb{R}^{+}$ we define $\mathcal{A}_{\varrho}:L^2([0,T];L^2(\Omega,\varrho^{-1}))\to L^2([0,T];L^2(\Omega,\varrho^{-1}))$ as 
\begin{align}
\mathcal{A}_\varrho\vec{a} = -\varrho\bm{Z}_2\star \vec{a}
\end{align} 
for $\vec{a}\in L^2([0,T];L^2(\Omega,\varrho^{-1}))$. Then by the assumptions \eqref{ass:V1_V2_in_W_1_inf} one checks that for each density field $\varrho(\vec{r},t)$ the linear operator $\mathcal{A}_\varrho$ is Hilbert Schmidt with $\|\mathcal{A}\|_{\text{HS}}^2 =  \sum_{n \in \mathbb{N}} \|\mathcal{A}[\vec{v}_k,\varrho]\|^2_{L^2(\Omega,\varrho^{-1})}<\infty$, where $\{\vec{v}_k\}_{k\in\mathbb{N}}$ is an orthonormal basis of $L^2([0,T];L^2(\Omega,\varrho^{-1}))$. Additionally, for each $\varrho(\vec{r},t)$, the matrix $\bm{D}_\varrho$ is a bounded, linear, finite rank operator as is therefore Hilbert-Schmidt. We now write $\mathcal{H}_\varrho = \bm{D}_\varrho^{-1}- \mathcal{A}_\varrho = \bm{D}_\varrho^{-1}\left(\bm{1}-\mathcal{Z}_\varrho\right)  : L^2(\Omega,\varrho^{-1})\to L^2(\Omega,\varrho^{-1})$, where $\mathcal{Z}_\varrho:= \bm{D}_\varrho\mathcal{A}_\varrho$. Now $\mathcal{Z}_\varrho = \bm{D}_\varrho\circ\mathcal{A}$ is the composition of two Hilbert-Schmidt operators, and is therefore trace class. By the classical theory \cite{fredholm1900nouvelle, lax2014functional}, we have the identity 
\begin{align}
\det(\mathcal{H}_\varrho) = \det(\bm{D}_\varrho^{-1}) \det(\bm{1}- \mathcal{Z}_\varrho) = \left( \prod_{j = 1}^d \mu_{j}^{-1}\right) \exp\Big\{-\sum_{n=1}^\infty\tfrac{\text{Tr}(\mathcal{Z}_\varrho^n)}{n}  \Big\}.\label{eq:det(1+Z)_in_terms_of_exp}
\end{align}
By positive definiteness of the matrix $\bm{D}_\varrho$ we have that $\prod_{j = 1}^d \mu_{j}^{-1}>0$. All that remains, therefore is to study the convergence of the infinite series in the exponential term in \eqref{eq:det(1+Z)_in_terms_of_exp}. We calculate the trace $\text{Tr}(\mathcal{Z}_\varrho^n)$ with respect to the basis $\{\vec{u}_k\}_{k = 1}^\infty$, thus
\begin{align}
\text{Tr}\mathcal{Z}_\varrho^n &= \sum_{k=1}^\infty\int^{T}_0\mathrm{d}t\,\langle \mathcal{Z}_\varrho^n \vec{u}_k, \vec{u}_k\rangle_{L^2(\Omega,\varrho^{-1})}=\sum_{k=1}^\infty |\gamma_k|^n\| \vec{u}_k\|^2_{L^2([0,T]; L^2(\Omega,\varrho^{-1}))} = \sum_{k=1}^\infty |\gamma_k|^n,
\end{align}
where we have used the fact that the $\vec{u}_k$ are orthonormal in $L^2([0,T]; L^2(\Omega,\varrho^{-1}))$. Now we return to \eqref{eq:det(1+Z)_in_terms_of_exp} and find
\begin{align}
\det(\bm{1}- \mathcal{Z}_\varrho) &=  \exp\Big\{-\sum_{n=1}^\infty\tfrac{\text{Tr}(\mathcal{Z}_\varrho^n)}{n}  \Big\} = \exp\Big\{-\sum_{n=1}^\infty\frac{1}{n}\sum_{k=1}^\infty|\gamma_k|^n  \Big\} = \exp\Big\{-\sum_{k=1}^\infty\sum_{n=1}^\infty\frac{1}{n}|\gamma_k|^n  \Big\}\nonumber\\
& = \exp\Big\{\sum_{k=1}^\infty\log \left(1-|\gamma_k| \right) \Big\} = \Pi_{k = 1}^\infty\left(1-|\gamma_k| \right).\label{eq:det_(1-lambdaZ)_prod}
\end{align}
only if $|\gamma_k   | <1$ for each $k\in\mathbb{N}$. We now estimate $\gamma_k$, by the definition of $\mathcal{Z}_\varrho$ we have
\begin{align}
&|\gamma_k|  = |\gamma_k|\int^T_0\mathrm{d}t\,\langle \vec{u}_k,\vec{u}_k\rangle_{L^2_{\varrho^{-1}}} = \int^T_0\mathrm{d}t\,\int\mathrm{d}\vec{r}\, \varrho^{-1}\vec{u}_k(\vec{r},t)\varrho(\vec{r},t)\bm{D}_\varrho\int\mathrm{d}\vec{r'}\, \bm{Z}_2(\vec{r},\vec{r}')\vec{u}_k(\vec{r}',t)\nonumber\\
& \leq \|\bm{Z}_2\|_{L^\infty}\|\mu_{\max}\|_{L^\infty([0,T])} \int^T_0\mathrm{d}t\,\left( \int\mathrm{d}\vec{r}\,|\vec{u}_k(\vec{r},t)|\right)^2 \\
&\leq \|\bm{Z}_2\|_{L^\infty} \|\mu_{\max}\|_{L^\infty([0,T])}\int^T_0\mathrm{d}t\,\left( \int\mathrm{d}\vec{r}\,\varrho^{1/2}|\varrho^{-1/2}\vec{u}_k(\vec{r},t)|\right)^2\nonumber\\
& \leq \|\bm{Z}_2\|_{L^\infty} \|\mu_{\max}\|_{L^\infty([0,T])}\int^T_0\mathrm{d}t\,\|\vec{u}_k\|^2_{L^2_{\varrho^{-1}}}  \int\mathrm{d}\vec{r}\,|\varrho(\vec{r},t)| \leq \|\mu_{\max}\|_{L^\infty([0,T])}\|\bm{Z}_2\|_{L^\infty},
\end{align}
where we have used the fact that $\|\varrho\|_{L^1} = 1$ (from Corollary \ref{cor:L_1_varrho_is_1}) and the orthonormality of $\vec{u}_k\in L^2([0,T];L^2(\Omega,\varrho^{-1}))$. Hence, for $\mu_{\max}\|\bm{Z}_2\|_{L^\infty}<1$ then $|\gamma_k  |<1$ and therefore $\det(\bm{1}- \mathcal{Z}_\varrho) \neq 0$. Thus lemma is proved.
\end{proof}

We now establish the spectral decomposition of the flux $\vec{a}(\vec{r},t)$ in terms of the eigenbasis $\vec{u}_k\in L^2([0,T];L^2(\Omega,\varrho^{-1}))$.

\begin{lemma}\label{prop:unique_a_flux_FAT}
Let $\|\mu_{\max}\|_{L^\infty([0,T])}\|\bm{Z}_2\|_{L^\infty} < 1$. Then for each density field $\varrho(\vec{r},t)>0$ there exists a unique flux field $\vec{a}(\vec{r},t)\in L^2(\Omega,\varrho^{-1})$ solving \eqref{eq:ddft_eq_dyn_a}.
\end{lemma}
\begin{proof}
By Lemma \ref{lem:H_lambda_is_compact_SA} we have that $\mathcal{H}_{\varrho}$ is a self-adjoint operator. Therefore, for a given density field $\varrho(\vec{r},t)$ the homogeneous adjoint problem may be written as
\begin{align}
\mathcal{H}_{\varrho}^\ast[\vec{z},\varrho] = \mathcal{H}_{\varrho}[\vec{z},\varrho] = \vec{0}.
\end{align}
By Proposition \ref{thm:cond_converg_fred_det} we have that $\mathcal{H}_{\varrho}$ is invertible with full rank, thus we have $\vec{z} = \vec{0}$ uniquely, and by the Fredholm Alternative there exists a unique solution $\vec{a}(\vec{r},  t)$ for each $\vec{r}\in \Omega$ and $t\in [0,\infty)$ to the inhomogeneous problem \eqref{eq:matrix_operator_friction_eqn}, equivalently \eqref{eq:ddft_eq_dyn_a}.
\end{proof}

The invertibility result from Proposition \ref{thm:cond_converg_fred_det} and the uniqueness result from Lemma \ref{prop:unique_a_flux_FAT} lets us present the following result on the implicit gradient flow form of the dynamics \eqref{eq:ddft_eq_dyn_rho}--\eqref{eq:ddft_eq_dyn_a}.

\begin{corollary}[Implicit Gradient Flow Form]{\ \\}
Let $\varrho(\vec{r},t)$ be a classical solution to \eqref{eq:ddft_eq_dyn_rho}, and suppose $\|\mu_{\max}\|_{L^\infty([0,T])}\|\bm{Z}_2\|_{L^\infty(\Omega)}<1$. Then, \eqref{eq:ddft_eq_dyn_rho}--\eqref{eq:ddft_eq_dyn_a} may be written in the closed implicit gradient flow form
\begin{align}
\partial_{t}\varrho = \nabla \cdot \left(\mathcal{H}^{-1}_{\varrho}\varrho \nabla \tfrac{\delta \mathcal{F}}{\delta \varrho}[\varrho]\right).\label{eq:gradient_flow_struct_rho}
\end{align}
\end{corollary}

\begin{remark}
We note that the invertibility of result Proposition \ref{thm:cond_converg_fred_det} does not require positive definiteness of $\mathcal{H}_\varrho$. Therefore the gradient flow structure \ref{eq:gradient_flow_struct_rho} would apply to apply to a wider variety of operators $\mathcal{H}_\varrho$, not necessarily carrying physically motivated properties from fluid mechanics, such as the strict positivity of the rate of mechanical energy dissipation.
\end{remark}

The following corollary provides a unique eigenfunction expansion for the flux $\vec{a}(\vec{r},  t)$ solving \eqref{eq:matrix_operator_friction_eqn} for each time dependent $\varrho(\vec{r},t)$. 

\begin{corollary}[Eigenfunction Expansion of the Flux $\vec{a}(\vec{r}, {[\varrho]}, t)$]\label{cor:eigenfn_expansion_of_flux}
For each density field $\varrho(\vec{r},t)>0$, we have the unique expression of the flux $\vec{a}(\vec{r},  t)\in L^2(\Omega,\varrho^{-1})$
\begin{align}
\vec{a}(\vec{r},  t) = \sum_{n\in\mathbb{N}}\gamma_n^{-1}\langle\vec{u}_n,\nabla \frac{\delta\mathcal{F}}{\delta\varrho}[\varrho]\rangle_{L^2(\Omega)} \vec{u}_n(\vec{r},t).
\end{align}
\end{corollary}

We now establish that stationary densities $\varrho_0(\vec{r})$ solving \eqref{eq:ddft_rho_stationary}--\eqref{eq:ddft_a_stationary} are in fact minimisers of the free energy functional $\mathcal{F}[\cdot]$ given in \eqref{eq:def_of_F}. Since $\varrho_(\vec{r})$ are independent of time we see, by \eqref{eq:def_diffusion_tensor}, that $\bm{D}_\varrho$ is also independent of time and hence each eigenvalue of $\bm{D}_\varrho$, $\mu_j$ is independent of time for $j = 1,\cdots, d$. Therefore, for the remaining results of the section it will be observed that $\|\mu_{j}\|_{L^\infty([0,T])} = \mu_j$, and in particular $\|\mu_{\max}\|_{L^\infty([0,T])} = \mu_{\max}$. Additionally, for $\mathcal{H}_\varrho$, we assume that in stationarity (i.e., as $t\to\infty$) we have the diagonalisation
\begin{align}
\lim_{t\to\infty}\mathcal{H}_\varrho\vec{u}(\vec{r},t) = \mathcal{H}_\varrho \vec{u}(\vec{r} ) = \gamma_k \vec{u}(\vec{r}),\label{eq:diagonalisation_of_H_rho_stationary}
\end{align}
where $\vec{u}(\vec{r})$ form an orthonormal basis of $L^2(\Omega,\varrho^{-1})$ (where $\varrho$ is a stationary density solving \eqref{eq:ddft_rho_stationary}--\eqref{eq:ddft_a_stationary}), and $\gamma_k$ are each independent of time for $k\in \mathbb{N}$. As such \eqref{eq:diagonalisation_of_H_rho_stationary} is the stationary counterpart to the diagonalistion \eqref{eq:diagonalisation_of_H_rho}.

\begin{proposition}[$\varrho(\vec{r})$ is a Critical Point of the Free Energy]\label{prop:association _of_free_energy}
\par
Let $\mu_{\max}\|\bm{Z}_2\|_{L^\infty(\Omega)}<1$ and $V_1 = V_1^{eq}(\vec{r})$ and $V_2 = V_2^{eq})$ be time independent functions so that $\varrho(\vec{r})>0$ is a stationary density to the system \eqref{eq:ddft_rho_stationary}--\eqref{eq:ddft_a_stationary}. Then $\varrho(\vec{r})$ is a critical point of $\mathcal{F}[\varrho]$, such that
\begin{align}
\frac{\delta\mathcal{F}}{\delta\varrho}[\varrho(\vec{r})]=\mu,\label{eq:euler-lagrange_eqn}
\end{align}
where $\mu$ is the chemical potential, a constant independent of both time and space.
\end{proposition}

\begin{proof}

Let $\varrho(\vec{r})$ be a stationary density. Then by equation \eqref{eq:ddft_a_stationary} one has
\begin{align}
\mathcal{H}_{\varrho}\vec{a}&=\varrho(\vec{r})\nabla \frac{\delta\mathcal{F}}{\delta\varrho}[\varrho](\vec{r}),\label{eq:stationary_eqn_for_a_dimensionless}\\
\nabla \cdot \vec{a}(\vec{r})&=0.\label{eq:incomp_a_at_equilibrium}
\end{align}
By Lemma \ref{lem:H_lambda_is_compact_SA} we have that the operator $\mathcal{H}_{\varrho}$ is compact self-adjoint in $L^2(\Omega,\varrho^{-1})$ (for $\varrho$ a stationary density solving \eqref{eq:ddft_rho_stationary}--\eqref{eq:ddft_a_stationary}) and by Proposition \ref{thm:cond_converg_fred_det} we have that $\mathcal{H}_{\varrho}$ is invertible. With this, and by using equation \eqref{eq:ddft_rho_stationary}, we have
\begin{align}\label{eq:div_stationary_eqn_in_terms_of_F}
0 = \nabla \cdot \vec{a} = \nabla \cdot \left(\mathcal{H}^{-1}_{\varrho}\varrho(\vec{r})\nabla \frac{\delta\mathcal{F}}{\delta\varrho}[\varrho](\vec{r})\right). 
\end{align}
Now since $\varrho$ is stationary a stationary density, we have that $\partial_t\varrho = 0$ and we see that
\begin{align}
0 = \Big\langle \frac{\delta\mathcal{F}}{\delta\varrho}[\varrho],\partial_t\varrho\Big\rangle_{L^2(\Omega)} = -\int \mathrm{d}\vec{r}\, \frac{\delta\mathcal{F}}{\delta\varrho}[\varrho]\nabla \cdot \vec{a} = \int \mathrm{d}\vec{r}\, \nabla \frac{\delta\mathcal{F}}{\delta\varrho}[\varrho]\cdot \vec{a}
\end{align}
where we have used the divergence theorem and the no-flux boundary condition \eqref{bc:DDFT} to eliminate the boundary integral. We now show that $\mathcal{H}_{\varrho}$ is strictly positive definite. In particular let $\vec{v}\in L^2(\Omega,\varrho^{-1})$ then 
\begin{align}
\langle \vec{v}(\vec{r}), \mathcal{H}_{\varrho}\vec{v}\rangle_{L^2_{\varrho^{-1}}} = \int \mathrm{d}\vec{r}\, \vec{v}(\vec{r})\cdot \bm{D}_\varrho \int \mathrm{d}\vec{r}'   \bm{Z}_2(\vec{r},\vec{r}')\vec{v}(\vec{r}') + \int \mathrm{d}\vec{r}\, \varrho^{-1} |\vec{v}(\vec{r})|^2.
\end{align}
We now bound the first integral below by using \eqref{ass:V1_V2_in_W_1_inf}, in particular
\begin{align}
\langle \vec{v}(\vec{r}), \mathcal{H}_{\varrho}\vec{v}\rangle_{L^2_{\varrho^{-1}}} &\geq -\mu_{\max}\|\bm{Z}_2\|_{L^\infty} \left(\int\mathrm{d}\vec{r}\, |\vec{v}(\vec{r})|\right)^2 + \int \mathrm{d}\vec{r}\, \varrho^{-1} |\vec{v}(\vec{r})|^2\nonumber\\
& \geq -\mu_{\max}\|\bm{Z}_2\|_{L^\infty}\int\mathrm{d}\vec{r}\, \varrho^{-1}|\vec{v}(\vec{r})|^2 + \int\mathrm{d}\vec{r}\, \varrho^{-1}|\vec{v}(\vec{r})|^2\\
 & = \left(1-\mu_{\max}\|\bm{Z}_2\|_{L^\infty}\right)  \|\vec{v}\|^2_{L^2_{\varrho^{-1}}} >0,
\end{align}
where we have used Jensen's inequality to bound the first term below and the fact that $\mu_{max}\|\bm{Z}_2\|_{L^\infty}<1$. Hence $\mathcal{H}_{\varrho}$ is strictly positive definite.

Now since $\mathcal{H}_{\varrho}$ is strictly positive definite and self-adjoint in $L^2(\Omega,\varrho^{-1})$ it possesses a unique strictly positive definite self-adjoint square root in $L^2(\Omega,\varrho^{-1})$ (see \cite{wouk1966note}) such that $\mathcal{H}_{\varrho} = \mathcal{H}_{\varrho}^{1/2}\mathcal{H}_{\varrho}^{1/2}$. Thus we find
\begin{align}
0 &= \int \mathrm{d}\vec{r}\, \nabla \frac{\delta\mathcal{F}}{\delta\varrho}[\varrho]\cdot \vec{a} = \int \mathrm{d}\vec{r}\, \nabla \frac{\delta\mathcal{F}}{\delta\varrho}[\varrho]\cdot \mathcal{H}_{\varrho} \varrho\nabla \frac{\delta\mathcal{F}}{\delta\varrho}[\varrho]\nonumber\\ 
&= \Big\langle \varrho \bm{D}\nabla \frac{\delta\mathcal{F}}{\delta\varrho}[\varrho],\mathcal{H}_{\varrho} \varrho\nabla \frac{\delta\mathcal{F}}{\delta\varrho}[\varrho]\Big\rangle_{L^2(\Omega,(\varrho\bm{D})^{-1})}=\Big\langle \mathcal{H}_{\varrho}^{1/2} \varrho\nabla \frac{\delta\mathcal{F}}{\delta\varrho}[\varrho],\mathcal{H}_{\varrho}^{1/2} \varrho\nabla \frac{\delta\mathcal{F}}{\delta\varrho}[\varrho]\Big\rangle_{L^2(\Omega,\varrho^{-1})}\nonumber\\
&=\Big\|\mathcal{H}_{\varrho}^{1/2} \varrho\nabla \frac{\delta\mathcal{F}}{\delta\varrho}[\varrho]\Big\|_{L^2(\Omega,\varrho^{-1})}^2,\label{eq:X_sqrt_energy_argument}
\end{align}
where we have used the self-adjoint property of $\mathcal{H}_{\varrho}^{1/2}$. From the above we deduce that, since the integrand in the last line of \eqref{eq:X_sqrt_energy_argument} is positive, that the stationary density $\varrho(\vec{r})$ satisfies the Euler-Lagrange equation
\begin{align}\label{eq:grad_F_rho_is_zero}
\nabla \frac{\delta\mathcal{F}}{\delta\varrho}[\varrho(\vec{r})] = \vec{0} \Rightarrow  \frac{\delta\mathcal{F}}{\delta\varrho}[\varrho(\vec{r})]-\mu =0
\end{align}
for a.e. $\vec{r}\in \Omega$, where $\mu$ is the chemical potential. Therefore we obtain that $\varrho$ is a critical point the free energy $\mathcal{F}[\varrho]$.
\end{proof}

From Proposition \ref{prop:association _of_free_energy} we obtain the following corollary. 

\begin{corollary}\label{cor:stationary_density_independent_of_D}

Let $\mu_{\max}\|\bm{Z}_2\|<1$ and let $\varrho(\vec{r})$ be a critical point of the free energy $\mathcal{F}$, then
\begin{enumerate}
\item $\varrho(\vec{r})$ is an equilibrium density of the system \eqref{eq:ddft_rho_stationary}--\eqref{eq:ddft_a_stationary}.
\item The equilibrium density $\varrho(\vec{r})$ is independent of $\bm{Z}_1$, $\bm{Z}_2$ and, as a consequence, of $\bm{D}$.
\item The equilibrium flux solving \eqref{eq:ddft_rho_stationary}--\eqref{eq:ddft_a_stationary} is $\vec{a}(\vec{r})\,!= \vec{0}$. 
\end{enumerate}
In particular, there do not exist stationary densities which are advected by the existence of some finite, nonhomogeous flux and hence the only stationary states of \eqref{eq:ddft_eq_dyn_rho}--\eqref{eq:ddft_eq_dyn_a} are equilibrium states.
\end{corollary}

We remark that Corollary \ref{cor:stationary_density_independent_of_D} is related to the well-known result that for finite dimensional reversible diffusions, i.e. Langevin dynamics of the form $dX_t = - D(X_t) \nabla V(X_t) \, dt + \nabla \cdot D(X_t) \, dt + \sqrt{2 D(X_t)} \, dW_t$ for an arbitrary strictly positive definite mobility matrix $D$, $V$ a confining potential and Wiener process $W_t$, the invariant measure $\mu(dx) = \frac{1}{Z} e^{-V(x)} \, dx$ is independent of $D$. We refer to \cite[Sec 4.6]{pavliotis2014stochastic}. To our knowledge, this is the first instance where such a result is proven in the context of DDFT. 
\\\\
\indent With the above results we establish the Lyapunov stability of the dynamics of $\varrho(\vec{r},t)$ solving \eqref{eq:ddft_eq_dyn_rho}--\eqref{eq:ddft_eq_dyn_a} with the following corollary.

\begin{corollary}[Lyapunov Stability for $\varrho(\vec{r},t)$]\label{cor:gradient_flow_structure}
Let $\|\mu_{\max}\|_{L^\infty([0,T])}\|\bm{Z}_2\|_{L^\infty} < 1$, then and let $\varrho(\vec{r},t)$ be a classical solution to \eqref{eq:ddft_eq_dyn_rho}-- \eqref{eq:ddft_eq_dyn_a}. Then, we have the following H-Theorem for $t\in [0,\infty)$:
\begin{align}
\der[]{t}\mathcal{F}[\varrho](t) = -\Big\| \mathcal{H}^{-1/2}_{\varrho} \varrho \nabla \frac{\delta\mathcal{F}}{\delta\varrho}[\varrho]\Big\|^2_{L^2_{\varrho^{-1}}}\leq 0.\label{eq:lyapunov_stability}
\end{align}
\end{corollary}

\subsection{Strict Positivity of $\varrho$.}\label{subsec:strict_pos_rho}
We now comment on the proof of strict positivity for classical solutions $\varrho(\vec{r},t)$ of \eqref{eq:ddft_eq_dyn_rho}. The nonnegativity of the solutions follows from Corollary \ref{cor:L_1_varrho_is_1}. Consider now the frozen version of \eqref{eq:ddft_eq_dyn_rho} of in gradient flow form (c.f. \eqref{eq:gradient_flow_struct_rho}), that is
\begin{align}
\partial_{t}\phi = \nabla \cdot \left(\mathcal{H}^{-1}_{\psi}\left(\nabla\phi + \phi\nabla V_1 + \phi\nabla V_2\star \varrho\right)  \right),\label{eq:gradient_flow_struct_rho_frozen}
\end{align}
for $\psi:\Omega\times [0,\infty) \to \mathbb{R}^+$. In particular since $\mathcal{H}_\psi$ is positive definite, and all coefficients in \eqref{eq:gradient_flow_struct_rho_frozen} are uniformly bounded. Additionally, $\phi(\vec{r},t) = \psi(\vec{r},t) = \varrho(\vec{r},t)$ is a classical solution to this PDE and thus we a Harnack inequality (c.f. \cite[ Theorems 8.1.1--8.1.3]{bogachev2015fokker} for sharp versions of this result) of the following form 
\begin{align}
\sup_{\vec{r}\in \Omega}\varrho(\vec{r},t_1)<C \inf_{\vec{r}\in \Omega}\varrho(\vec{r},t_2)
\end{align}
for $0<t_1<t_2<\infty$ and $C$ is a constant depending on $d$ (the dimension) and $\mu_{\max}$. Since $\varrho$ is non-negative for all time we must have $\inf_{\vec{r}\in \Omega}\varrho(\vec{r},t)$ is positive and hence $\varrho$ is positive. 

\subsection{Characterisation of Stationary Solutions}\label{sec:char_stationary_sol}\label{sec:stationary_problem}
In this section we consider the global asymptotic stability of the stationary equations \eqref{eq:ddft_rho_stationary}--\eqref{eq:ddft_a_stationary}. 

A particular solution of the stationary equations \eqref{eq:ddft_rho_stationary}--\eqref{eq:ddft_a_stationary} satisfies the following fixed point equation, with zero flux
\begin{subequations}{}
\begin{align}
&\varrho_0 = \frac{e^{-\left( V_1^{eq}+ V_2^{eq}\star\varrho_0\right)}}{Z[\varrho_0]}, 
\label{eq:self_cons_eq}\\
&\vec{a}_0(\vec{r}) = \vec{0},
\end{align}
\end{subequations}
where $Z[\varrho_0] := \int  \mathrm{d}\vec{r}\,e^{-\left( V_1^{eq}+ V_2^{eq}\star\varrho_0 \right)}$ and, notably, $\varrho_0$ is independent of the HI tensors. 

We have by Proposition \ref{prop:association _of_free_energy} that equilibrium densities satisfy the Euler-Lagrange equation \eqref{eq:euler-lagrange_eqn}. Furthermore, we have by Corollary \ref{cor:stationary_density_independent_of_D}, that in equilibrium that $\vec{a}(\vec{r}) \, != \vec{0}$. Therefore by Corollary \ref{cor:stationary_density_independent_of_D} we find that $\mathcal{H}_{\varrho}^{-1} \equiv \bm{1}$ in equilibrium and by \eqref{eq:ddft-eq-gradient_flow_form} we find that equilibrium densities $\varrho_0$ satisfy 
\begin{subequations}
\begin{align}
\nabla \cdot \left(\bm{D}_{\varrho_0}\left( \nabla \varrho_0+\varrho_0\nabla ( V_1^{eq}+ V_2^{eq}\star\varrho_0)\right) \right)   &= 0 \qquad \vec{r}\in \Omega, \label{eq:stationary_rho_problem}\\
\bm{D}_{\varrho_0}\left( \nabla \varrho_0+\varrho_0\nabla ( V_1^{eq}+  V_2^{eq}\star\varrho_0)\right)\cdot \vec{n} &= 0 \qquad \vec{r} \text{ on } \partial\Omega. \label{eq:stationary_rho_problem_bc}
\end{align}
\end{subequations}
We seek classical solutions $\varrho_0\in C^2(\Omega)$ to \eqref{eq:stationary_rho_problem}--\eqref{eq:stationary_rho_problem_bc}. The existence and uniqueness for this stationary problem is based on a fixed point argument for the nonlinear map, defined by integrating equation \eqref{eq:stationary_rho_problem}. In particular we find the stationary distribution satisfies the self-consistency equation \eqref{eq:self_cons_eq}. We now present the first result concerning the existence and uniqueness of the solutions to the self-consistency equation \eqref{eq:self_cons_eq}.
\\
\begin{lemma}[Existence and Uniqueness of Stationary Solutions]\label{thm:exis_fix_point}
For the equilibrium problem \eqref{eq:stationary_rho_problem}--\eqref{eq:stationary_rho_problem_bc} we have the following
\begin{enumerate}
\item The stationary equation \eqref{eq:stationary_rho_problem} with boundary condition \eqref{eq:stationary_rho_problem_bc} has a smooth, non-negative solution $\varrho_0\in L^1(\Omega)$ with $\|\varrho_0\|_{L^1(\Omega)}=1$.
\item When the interaction energy is sufficiently small, $\|V_2\|_{L^\infty}\leq 1/4$, the solution $\varrho_0$ is unique.
\end{enumerate}
\end{lemma}

\begin{proof}
The proof follows Dressler et al. \cite{dressler1987stationary}. The main idea is to show that the right hand side of equation \eqref{eq:self_cons_eq} is a contraction map on $C^2(\Omega)$ and the proof is included in Appendix \ref{thm:exis_fix_point}.
\end{proof}

\begin{proposition}[Existence, Regularity, and Strict Positivity of Solutions for the Stationary Problem]
Consider the stationary problem \eqref{eq:ddft_a_stationary}, equivalently \eqref{eq:stationary_rho_problem} such that
Assumption \eqref{ass:V1_V2_in_H_1} holds. Then we have that
\begin{enumerate}
\item There exists a weak solution $\varrho_0\in H^1(\Omega)\cap P_{ac}(\Omega)$ to \eqref{eq:stationary_rho_problem} as a fixed point of the equation \eqref{eq:self_cons_eq}.
\item Any weak solution  $\varrho_0\in H^1(\Omega)\cap P_{ac}(\Omega)$ is smooth and strictly positive, that
is $\varrho_0\in C^{\infty}(\bar{\Omega})\cap P^+_{ac}(\Omega)$.
\end{enumerate}
\end{proposition}
\begin{proof}

For a proof see \cite{GMPEQDDFT21}.
\end{proof}

We can also obtain an estimate on the rate of convergence to the equilibrium density in $L^2(\Omega)$ as $t\to \infty$ with the following theorem. In order to forgo additional assumptions on the initial data $\varrho_0$, e.g. partitioning the space of initial densities into basins of attraction, we restrict ourselves to the case where the equilibrium density $\varrho_0$ is unique.
\\
\begin{theorem}[Trend to Equilibrium in $L^2(\Omega)$]\label{lem:exp_conv_in_L2}
\par
Let $\varrho\in C^1([0,\infty];C^2(\Omega)) $ be a solution of \eqref{eq:ddft_dyn_non_gradient_flow} with initial data $\varrho_0\in L^2(\Omega)$ a probability density. Let $1 \leq 1/4\|V_2^{eq}\|_{L^\infty}^{-1}$, if
\begin{align}
r_{t} &:= \hat{\mu}_{\min}(t)c^{-2}_{pw}- 2\hat{\mu}_{\max}(t) \left(\|\nabla V_1^{eq}\|_{L^\infty(\Omega)}^2+ (e+1)\|\nabla V_2^{eq}\|_{L^{\infty}(\Omega)}^2\right)\\
&\quad - \hat{\mu}_{\max}(t)\|\bm{Z}_2\|_{L^\infty(\Omega)}^2\|\vec{a}\|_{L^2([0,T];L^1(\Omega))}^2>0,
\end{align} 
where $c_{pw}$ is a Poincar{\' e}$-$Wirtinger constant on the domain $\Omega$ and $\hat{\mu}_{\max}(t) = \int_0^t\mathrm{d}s\,\mu_{\max}(s)$ and $\hat{\mu}_{\min}(t) = \int_0^t\mathrm{d}s\,\mu_{\min}(s)$ are the time-mean of the largest and smallest eigenvalues of $\bm{D}_{\varrho(\vec{r},t)}$. Then, $\varrho\to \varrho_\infty\in C^\infty(\Omega)\cap P_{ac}^{+}(\Omega)$ in $L^2(\Omega)$ exponentially as $t\to \infty$, where $\varrho_\infty$ is the unique equilibrium density ensured by Lemma \ref{thm:exis_fix_point}. In particular the convergence in $L^2(\Omega)$ is given by
\begin{align}
\|\varrho(\cdot, t)-\varrho_\infty(\cdot)\|^2_{L^2(\Omega)}\leq \|\varrho_0(\cdot)-\varrho_\infty(\cdot)\|_{L^2(\Omega)}^2 e^{-r_{ t}}
\end{align}
as $t \to \infty$. 
\end{theorem}

\begin{proof}
For a proof we refer the reader to Theorem \ref{lem:exp_conv_in_L2_app}, which determines the convergence in the case $V_1^{eq} = \bm{Z}_2 = 0$. When both $V_1^{eq} \neq 0$ and $\bm{Z}_2 \neq 0$, their inclusion are linear in the PDEs \eqref{eq:ddft_dyn_non_gradient_flow}, \eqref{eq:stationary_rho_problem} so the only term to resolve for the evolution equation for $\psi$ first occurring at \eqref{eq:exp_conv_in_l2_bound_1} being
\begin{align}
\|\psi\bm{D}_{\varrho}^{1/2}\vec{v}_{1}\|_{L^2(\Omega)}^2 \leq 2\left(  \hat{\mu}_{\max}(t)\|\nabla V_1^{eq}\|_{L^\infty(\Omega)}^2+\max{\mu}_{\max}(t)\|\bm{Z}_2\|^2_{L^\infty} \|\vec{a}\|^2_{L^1}\right) \|\psi\|_{L^2(\Omega)}^2,
\end{align}
where $\vec{v}_1 = \nabla V_1^{eq} + \bm{Z}_2\star\vec{a}$
and we have used the fact that $\|\varrho\|_{L^1} = 1$ (see Corollary \ref{cor:L_1_varrho_is_1}) for all $t>0$. The remainder of the calculations to derive a Gr{\"o}nwall type inequality including this term are similar. 
\end{proof}


We remark that $\psi\in \left\lbrace u\in H^1(\Omega)\, | \,\int \mathrm{d}\vec{r}\,u = 0 \right\rbrace $, therefore, we may determine that the sharpest value of $c_{pw}$ conincides with the Poincar{\'e} constant as found by Steklov \cite{kuznetsov2015sharp}, equal to $\nu_1^{-1/2}$ where $\nu_1$ is the smallest eigenvalue of the problem
\begin{align}
\Delta u &= -\nu u \quad \text{ in } \Omega,\\
\partial_{\vec{n}}u & = 0 \qquad \, \text{ on } \partial \Omega.
\end{align}
Here $\partial_{\vec{n}}$ is the directional derivative along the unit vector $\vec{n}$ pointing out of the domain $\Omega$. Additionally Payne and Weinberger \cite{payne1960optimal} proved that for convex domains in $\mathbb{R}^n$ one has $c_{pw}\leq \tfrac{\mathrm{diam}(\Omega)}{\pi}$.
\\
\section{Discussion \& Open Problems}\label{sec:discussion}

In this paper, the global well-posedness of overdamped DDFT \eqref{eq:ddft_eq_dyn_rho}--\eqref{eq:ddft_eq_dyn_a}. The results that we have presented reinforce the mathematical foundations of DDFT, and provide greater theoretical justification for in computational applications. Particuarly, these equations govern the evolution of a collection of colloidal particles subject to both static inter-particle interactions and hydrodynamic interactions in a homogeneous background bath in the overdamped (high friction) limit and was derived rigorously as a solvability condition of the corresponding Vlasov-Fokker-Planck equation for the one-body density in position and momentum space $f(\vec{r},\vec{p},t)$ \cite{goddard2012overdamped}. In addition, $\bm{Z}_1$ and $\bm{Z}_2$ are two body nondimensional HI tensors which couple the momenta of the colloidal particles to HI forces on the same particles, mediated by fluid flows in the bath. The functions $V_1$ and $V_2$ are one and two body potentials respectively for static particle interactions. 

We first showed that the density and flux exist and are unique in the weak sense, \eqref{eq:weak_formulation_pair_rho}--\eqref{eq:weak_formulation_pair_a}, with no-flux boundary conditions, under sensible assumptions on the confining and interaction potentials and initial data $V_1$, $V_2$ and $\varrho(\vec{r},0)$ respectively. Additionally, we have derived a variational principle for the time dependent flux, which states that weak flux solutions  coincide with minimisers of the carefully defined functional $\mathcal{J}[\cdot]$, and vice-versa. This result is particularly relevant in the context of Power Functional Theory (PFT) \cite{schmidt2013power,schmidt2018power}, wherein superadiabatic forces (which are usually neglected in the formal derivation of a physical DDFT) are accounted for, by considering the colloid flux $\vec{a}(\vec{r},t)$ in a variational principle, and as the main quantity to be derived (over, e.g., the density $\varrho(\vec{r},t)$). 

\emph{Prima facie}, due to the added complexity of the HIs, the dynamics \eqref{eq:ddft_eq_dyn_rho}--\eqref{eq:ddft_eq_dyn_a} appeared to not be expressible in gradient flow form other than under the simplifying assumption $\bm{Z}_2 = 0$, or both $\bm{Z}_1 = \bm{Z}_2 = 0$. However, in this paper we showed that a gradient flow form does indeed exist under reasonable assumptions on the two-body tensors, in particular the $L^\infty(\Omega)$ assumptions on the individual HI tensors, \ref{ass:Z2_uniformly_bd}, are permissible by reference to the rigorous derivation of the classical boundary value problems in the fluid mechanics of two body hydrodynamics, i.e., Oseen \cite{kim2013microhydrodynamics} and \cite{rotne1969variational} tensors. Additionally, it has rigorously been previously shown that (see \cite{goddard2012overdamped}) that for $\bm{Z}_1$ being positive definite, $\bm{D}$ is also positive definite. 

Additionally for classical solutions, the gradient flow form was principally achieved by establishing the invertibility of the operator $\mathcal{H}_{\varrho}$, and, in doing so, we showed that the functional $\mathcal{F}[\varrho]$ was uniquely associated to the PDE \eqref{eq:ddft_eq_dyn_rho}, \eqref{eq:ddft_eq_dyn_a} for the general case $\bm{Z}_1$, $\bm{Z}_2 \neq 0$, which establishes for the first time a stationary action principle for the dynamics of the density converging equilibrium including HIs. The overarching accomplishment was to show that $\partial_t \varrho = 0$ implies $\varrho$ is a critical point of $\mathcal{F}$, as would naturally be expected, and that $\vec{a}(\vec{r},[\varrho]) = \vec{0}$ is the unique steady state flux. The gradient flow structure in the evolution equation \eqref{eq:ddft-eq-gradient_flow_form} provides $\mathcal{F}$ as a Lyapunov function. The mechanism to do this is related to the fact that $\mathcal{H}^{-1}_\varrho$ can be seen as the resolvent of the operator $\mathcal{Z}_\varrho[\cdot] = \varrho\bm{D}\bm{Z}_2\star[\cdot]$, which amalgamates the nonlocal effects (diffusion $\bm{D}$ and drift $\bm{A}$) of the HIs into a single action on the density. The compactness, self-adjointness and subsequently conditional positivity of $\mathcal{H}_{\varrho}$ then provide a formula for the time derivative of $\mathcal{F}$ involving the action of $\bm{D}$ and $\bm{A}$, and, by Corollary \ref{cor:gradient_flow_structure} we proved an H-theorem establishing a monotonically decreasing free energy in the transient density $\varrho(\vec{r},t)$ to equilibrium, including the dynamic effects of HIs. 
We note that the invertibility of $\mathcal{H}_\varrho$ was established without the need for positive definiteness, and as such, Proposition \ref{thm:cond_converg_fred_det} would be applicable to a wider variety of operators $\mathcal{H}_\varrho$ not necessarily carrying physically motivated properties (i.e., that for a particle undergoing friction in a thermostated bath, that the rate of mechanical energy dissipation should necessarily be positive). 

Assuming a classical solution to the DDFT we also derived \emph{a priori} convergence estimates in $L^2$ and relative entopy, the latter restricted to convex two-body potentials. Well-posedness and global asymptotic stability of the phase space equation for the time evolution of $f(\vec{r},\vec{p},t)$ remains open (see \cite[Proposition 2.1]{goddard2012overdamped} for the evolution equation for $f(\vec{r},\vec{p},t)$). It is of similar form to the Vlasov equation considered by \cite{degond1986global} but with Hermite dissipative term and modified nonlocal term in the momentum variable $\vec{p}$ dependent on the HI tensors. To progress further some maximum principles on $f(\vec{r},\vec{p},t)$ solving the linearised version of the phase space equation must be found. Additionally, the existence results on the overdamped equations considered here may be made more regular by routine arguments. 

We also note that the present analysis is based on the Smoluchowski equation rigorously derived from the phase space Fokker-Planck equation using homogenisation methods \cite{goddard2012overdamped}. As an alternative to this, assuming inertia is small altogether, or if one is interested only in very short times to begin with, the system of interacting particles maybe considered solely in configuration space. Only the positions (and not the momenta) of a system of interacting Brownian particles are then taken into account with Smoluchowski equation as in \cite{rex2009dynamical}, and, the underlying Langevin dynamics contain only velocity equations for each particle which are usually written down \emph{a posteriori}. The justification for this is that the momentum distribution is assumed to have a minor role in the dynamical description of the fluid density, and indeed is taken to be irrelevant at the microscopic level. This Brownian approximation may also hold for highly dense suspensions, since in dense Newtonian systems there is a fast transfer of momentum and kinetic energy from the particle collisions, and this effect may be accounted for most efficiently by the bath in the Brownian dynamics with a non constant diffusion tensor. 
It is known however that the one-body Smoluchowski equation in \cite{rex2009dynamical} does not equate to equations \eqref{eq:ddft_eq_dyn_rho}-\eqref{eq:ddft_eq_dyn_a} which are obtained in the rigorous overdamped limit starting from the Newtonian dynamics. Intuitively this is because the two-body assumption for the HI ($\boldsymbol{\Gamma}$) and mobility ($\bm{D}$) tensors and the matrix inversion $\bm{D} = \boldsymbol{\Gamma}^{-1}$ are not commutable operations; even if $\bm{D}$ is two-body then $\text{det}(\bm{D})$ is not. A flow chart demonstrating the permitted commutations between various formalisms is included in \cite{goddard2012overdamped}. The nonequivalence of the two Smoluchowski equations is not considered here, and therefore a natural extension for future work would be to determine the existence, uniqueness and regularity of of the density starting from \cite{rex2009dynamical} as well as the corresponding conditions for linear stability.

\appendix
\section{Properties of $\mathcal{H}_{\varrho}$}

\begin{lemma}\label{lem:H_lambda_is_compact_SA}
Let $\varrho(\vec{r},t):\Omega\times [0,\infty)$ and $\mathcal{H}_{\varrho}:L^2([0,T]; L^2(\Omega,\varrho^{-1}))\to L^2([0,T]; L^2(\Omega,\varrho^{-1}))$ be defined by \eqref{def:H_phi_op}. Then, $\mathcal{H}_{\varrho}$ is a compact and self-adjoint operator in $L^2(\Omega,\varrho^{-1})$. Furthermore, we let $\mathcal{H}_{\varrho}$ be diagonalised by 
\begin{align}
\mathcal{H}_{\varrho}\vec{u}_k = \gamma_k\vec{u}_k
\end{align}
where $\{\vec{u}_k\}_{k = 1}^\infty$ where $\vec{u}_k$ form an orthonormal basis of $L^2([0,T]; L^2(\Omega,\varrho^{-1}))$.
\end{lemma}
\begin{proof}
We let $\mathcal{Z}_\varrho = \varrho\bm{Z}_2\star[\cdot]$ and see that $\mathcal{H}_\varrho = \bm{D}_\varrho^{-1}+\mathcal{Z}_\varrho$. It suffices to check the symmetry of the operator $\mathcal{Z}_\varrho$. Let $\vec{f},\vec{h}\in L^2([0,T]; L^2(\Omega,\varrho^{-1}))$, we see by direct calculation
\begin{align}
\int_0^T\mathrm{d}t\,\langle \vec{h}, \mathcal{Z}_\varrho\vec{f}\rangle_{L^2(\Omega,\varrho^{-1})} &= \int_0^T\mathrm{d}t\,\int\mathrm{d}\vec{r}\, \varrho^{-1}(\vec{r},t) \vec{h}(\vec{r},t)\cdot\varrho(\vec{r},t) \int\mathrm{d}\vec{r}'\,  \bm{Z}_{2}(\vec{r},\vec{r}')\vec{f}(\vec{r}',t)\nonumber\\
& = \int_0^T\mathrm{d}t\,\int\mathrm{d}\vec{r}'\, \varrho^{-1}(\vec{r}',t)\vec{f}(\vec{r}',t)\cdot \varrho(\vec{r}',t)\int\mathrm{d}\vec{r}\, \vec{h}(\vec{r},t) \bm{Z}_{2}(\vec{r},\vec{r}')\nonumber\\
& = \int_0^T\mathrm{d}t\,\langle \mathcal{Z}_\varrho^\ast[\vec{h},\varrho], \vec{f} \rangle_{L^2(\Omega,\varrho^{-1})} \label{eq:adjoint_of_Z}
\end{align}
where in the final line we have used Fubini's theorem to interchange the order of integration. By the implied definition of $\mathcal{Z}_\varrho^\ast$ in \eqref{eq:adjoint_of_Z} we see that $\mathcal{Z}_\varrho^\ast = \mathcal{Z}_\varrho$ and hence $\mathcal{Z}_\varrho$ is self-adjoint. For the compactness, note that
\begin{align}
\int_0^T\mathrm{d}t\,\int\mathrm{d}\vec{r}\,\int\mathrm{d}\vec{r}'\,\varrho^{-1}(\vec{r},t) \varrho^2(\vec{r},t)\bm{Z}_{2}^2(\vec{r},\vec{r}') &\leq \|\mu_{\max}\|_{L^\infty([0,T])} \|\bm{Z}_2\|_{L^\infty}^2\\
&\quad \times |\Omega| \cdot \int_0^T\mathrm{d}t\,\|\varrho(\cdot,t)\|_{L^1(\Omega)} < \infty,
\end{align}
where we have used \eqref{ass:Z2_uniformly_bd}. Hence $\mathcal{Z}_\varrho$ is a Hilbert--Schmidt integral operator and is therefore compact. Thus Lemma is proved. 
\end{proof}

\section{Results on Convergence to Equilibrium}

\begin{lemma}[Existence and Uniqueness of Stationary Solutions]\label{thm:exis_fix_point}
\par
The self consistency equation \eqref{eq:self_cons_eq} has a smooth solution with $\|\varrho\|_{L^1(\Omega)}=1$. When the interaction energy is sufficiently small, $1 \leq 1/4\times \|V_2^{eq}\|_{L^\infty}^{-1}$, the solution is unique.
\end{lemma}
\begin{proof}
The proof follows Dressler et al. \cite{dressler1987stationary}. The main idea is to show that the right hand side of equation \eqref{eq:self_cons_eq} is a contraction map on $\varrho\in L^1$ a non-negative function with unit mean. 
Let $\Upsilon(\vec{r},[\varrho]) =  V_1^{eq}(\vec{r})+  V_2^{eq}\star\varrho(\vec{r})$. We show that the map $S:L^1(\Omega)\to L^1(\Omega)$ defined by
\begin{align}
S\varrho(\vec{r}):=\frac{1}{Z(\varrho)}\exp\left\lbrace - \Upsilon(\vec{r},[\varrho]) \right\rbrace
\end{align}
has a fixed point. Let $B$ be the unit ball in $L^1(\Omega)$ (the subset of pre-normalised functions) we clearly have $S(B)\subset B$ since $\|S\varrho\|_{L_1(\Omega)} = 1$. We must show that $S$ is continuous and $S(B)$ is a compact subset of $B$. Observe that if $V_2^{eq}$ is uniformly continuous then $\{V_2^{eq}\star \varrho \,| \, \varrho\in B\}$ is uniformly equicontinuous. Then by Arzela-Ascoli there exists a sequence $V_2^{eq}\star \varrho_{n_k}$ converging uniformly to some $F$,
\begin{align}
V_2^{eq}\star \varrho_{n_k} \to F \quad \text{ in } L^\infty(\Omega) \quad \text{as} \quad  k\to \infty.
\end{align}

Now observe that there exists $N$ such that for every $n_k \geq N$
\begin{align}\label{eq:arzela-ascoli_limit_integral}
\int\mathrm{d}\vec{r}\,|e^{-( V_1^{eq}+ V_2^{eq}\star \varrho_{n_k})}-e^{-( V_1^{eq}+ F)}| \leq \tfrac{1}{2}\int_\Omega\mathrm{d}\vec{r}e^{- V_1^{eq}(\vec{r})}.
\end{align}
So, by the Lebesgue dominated convergence theorem, since the integral may be dominated by constants times $e^{- V_1^{eq}(\vec{r})}$ and the limit $k\to \infty$ may be taken inside the left hand side integral of \eqref{eq:arzela-ascoli_limit_integral} giving
\begin{align}
\lim_{k\to \infty}\int\mathrm{d}\vec{r}\,|e^{-( V_1^{eq}+ V_2^{eq}\star \varrho_{n_k})}-e^{-( V_1^{eq}+ F)}| = 0.
\end{align}
Similarly the composition of $\exp(\cdot)$ and $V_2^{eq}\star \varrho$ is continuous and by the Lebesgue dominated convergence theorem
\begin{align}
\lim_{n\to\infty} Z(\varrho_{n}) = Z(\lim_{n\to\infty} \varrho_{n})= Z(F).
\end{align}
Hence $S$ is continuous. Now we may write
\begin{align}
S\varrho_{n}\to f:=\frac{e^{-(F+ V_1^{eq})}}{Z(F)} \quad \text{ in } L^1 \quad \text{ as } n\to \infty.
\end{align}
Hence for any sequence in $S(B)$ there is a convergent subsequence whose limit is in $S(B)$ and $\text{Im}(S)$ is compact. So by Schauder fixed point theorem there exists a fixed point.

 Now let $\varrho_1,\varrho_2\in B$ then
\begin{align}
\|S\varrho_1-S\varrho_2\|_{L^1}
& =\int \mathrm{d}\vec{r} \Big| \tfrac{e^{-\Upsilon[\varrho_1]}}{Z(\varrho_1)} - \tfrac{e^{-\Upsilon[\varrho_2]}}{Z(\varrho_2)} \Big|  
= \int \mathrm{d}\vec{r} \Big| \tfrac{e^{-\Upsilon[\varrho_1]}}{Z(\varrho_1)} - \tfrac{e^{-\Upsilon[\varrho_2]}}{Z(\varrho_1)}
+\tfrac{e^{-\Upsilon[\varrho_2]}}{Z(\varrho_1)} - \tfrac{e^{-\Upsilon[\varrho_2]}}{Z(\varrho_2)}\Big| \nonumber \\
&\leq \tfrac{1}{Z(\varrho_1)} \int \mathrm{d}\vec{r} \Big| e^{-\Upsilon[\varrho_1]} - e^{-\Upsilon[\varrho_2]}\Big| 
+ \int \mathrm{d}\vec{r}\Big|\tfrac{e^{-\Upsilon[\varrho_2]}}{Z(\varrho_1)}-\tfrac{e^{-\Upsilon[\varrho_2]}}{Z(\varrho_2)} \Big|.\label{estimate41a}
\end{align}
Considering now the second term of~\eqref{estimate41a}, we have
\begin{align}
\int \mathrm{d}\vec{r} \, \Big|\tfrac{e^{-\Upsilon[\varrho_2]}}{Z(\varrho_1)}-\tfrac{e^{-\Upsilon[\varrho_2]}}{Z(\varrho_2)} \Big|
& =  \Big|\tfrac{1}{Z(\varrho_1)}-\tfrac{1}{Z(\varrho_2)} \Big| \int \mathrm{d}\vec{r} \,  e^{-\Upsilon[\varrho_2]}
= Z(\varrho_2) \Big|\tfrac{1}{Z(\varrho_1)}-\tfrac{1}{Z(\varrho_2)} \Big| \\
& = \Big|\tfrac{Z(\varrho_2) - Z(\varrho_1)}{Z(\varrho_1)}\Big|
= \tfrac{1}{Z(\varrho_1)} \Big| \int \mathrm{d}\vec{r}\,e^{-\Upsilon[\varrho_1]}-\int \mathrm{d}\vec{r}\,e^{-\Upsilon[\varrho_2]}  \Big| \\
& \leq \tfrac{1}{Z(\varrho_1)}  \int \mathrm{d}\vec{r}\, \Big|e^{-\Upsilon[\varrho_1]} - e^{-\Upsilon[\varrho_2]}  \Big|.
\end{align}
Using this estimate in ~\eqref{estimate41a} then gives
\[
	\|S\varrho_1-S\varrho_2\|_{L^1} \leq \tfrac{2}{Z(\varrho_1)}  \int \mathrm{d}\vec{r}\, \Big|e^{-\Upsilon[\varrho_1]} - e^{-\Upsilon[\varrho_2]}  \Big|.
\]
We will now show that $S$ is a contraction. We have, by the mean value theorem, $\forall a,b\in\mathbb{R}$, $|e^{a}-e^{b}|\leq e^{a}e^{|a-b|}|a-b|$. Using this inequality with $a = -\Upsilon[\varrho_1]$, $b = -\Upsilon[\varrho_2]$ gives
\begin{align}
\tfrac{2}{Z(\varrho_1)}\int \mathrm{d}\vec{r}\,|e^{-\Upsilon[\varrho_1]}-e^{-\Upsilon[\varrho_2]}|  
&\leq \tfrac{2}{Z(\varrho_1)} \int \mathrm{d}\vec{r}\, e^{-\Upsilon[\varrho_1]} e^{|\Upsilon[\varrho_1] - \Upsilon[\varrho_1]|}|\Upsilon[\varrho_1] - \Upsilon[\varrho_1]| \\
&= \tfrac{2}{Z(\varrho_1)} \int \mathrm{d}\vec{r}\, e^{-\Upsilon[\varrho_1]} e^{|  V_2^{eq} \star (\varrho_1-\varrho_2)|}
|  V_2^{eq} \star (\varrho_1-\varrho_2)|.
\end{align}
Note that
\[
	|V_2^{eq} \star f| = \Big| \int \mathrm{d}\vec{r}' V_2^{eq}(\vec{r}-\vec{r'}) f(\vec{r}') \Big|
	\leq \| V_2^{eq}\|_{L^\infty} \Big| \int \mathrm{d}\vec{r'} f(\vec{r}')  \Big|
	\leq \| V_2^{eq}\|_{L^\infty} \| f \|_{L^1},
\]
and assuming $ 1 \leq1/4\times \|V_2^{eq}\|_{L^\infty}^{-1}$, we obtain
\begin{align}
\tfrac{2}{Z(\varrho_1)}\int \mathrm{d}\vec{r}\,|e^{-\Upsilon[\varrho_1]}-e^{-\Upsilon[\varrho_2]}|  
&\leq \tfrac{2}{Z(\varrho_1)} e^{\tfrac{1}{4} \|\varrho_1-\varrho_2\|_{L^1}} \tfrac{1}{4} \|\varrho_1-\varrho_2\|_{L^1} 
\int \mathrm{d}\vec{r}\, e^{-\Upsilon[\varrho_1]} \\
& \leq \frac{e^{1/2}}{2}  \|\varrho_1-\varrho_2\|_{L^1} 
< \|\varrho_1-\varrho_2\|_{L^1},
\end{align}
where we have used that  $\|\varrho_1-\varrho_2\|_{L^1} \leq 2$ and $e^{1/2}/2<1$.
Hence $S$ is a contraction and by the contraction mapping theorem the fixed point is unique.
\end{proof}

\begin{theorem}[Trend to Equilibrium in $L^2(\Omega)$]\label{lem:exp_conv_in_L2_app}
\par
Let $\varrho\in C^1([0,\infty];C^2(\Omega)) $ be a solution of \eqref{eq:ddft_dyn_non_gradient_flow} with initial data $\varrho_0\in L^2(\Omega)$ a probability density. Let $1 \leq 1/4\|V_2^{eq}\|_{L^\infty}^{-1}$, if
\begin{align}
r_{t} := \hat{\mu}_{\min}(t)c^{-2}_{pw}- 2\hat{\mu}_{\max}(t)(e+1)\|\nabla V_2^{eq}\|_{L^{\infty}(\Omega)}^2>0,
\end{align} 
where $c_{pw}$ is a Poincar{\' e}$-$Wirtinger constant on the domain $\Omega$ and $\hat{\mu}_{\max}(t) = \int_0^t\mathrm{d}s\,\mu_{\max}(s)$ and $\hat{\mu}_{\min}(t) = \int_0^t\mathrm{d}s\,\mu_{\min}(s)$ are the time-mean of the largest and smallest eigenvalues of $\bm{D}_{\varrho(\vec{r},t)}$. Then, $\varrho\to \varrho_\infty\in C^\infty(\Omega)\cap P_{ac}^{+}(\Omega)$ in $L^2(\Omega)$ exponentially as $t\to \infty$, where $\varrho_\infty$ is the unique equilibrium density ensured by Lemma \ref{thm:exis_fix_point}. In particular the convergence in $L^2(\Omega)$ is given by
\begin{align}
\|\varrho(\cdot, t)-\varrho_\infty(\cdot)\|^2_{L^2(\Omega)}\leq \|\varrho_0(\cdot)-\varrho_\infty(\cdot)\|_{L^2(\Omega)}^2 e^{-r_{ t}}
\end{align}
as $t \to \infty$. 
\end{theorem}

\begin{proof}
Let $\psi  = \varrho-\varrho_\infty$, then the evolution equation for $\psi$ may be written 
\begin{align}\label{eq:exp_conv_in_l2_bound_1}
\partial_t\psi - \nabla \cdot(\bm{D}_{\varrho}\,\nabla \psi) =   \nabla \cdot(\bm{D}_{\varrho}\,(\varrho_{\infty}\nabla V_2^{eq}\star \psi +\psi\nabla V_2^{eq}\star \varrho)).
\end{align}
Multiplying by $\psi$, integrating and using the boundary condition $\Pi[\psi]\cdot\vec{n} = 0$ on $\partial \Omega\times [0,T]$ we obtain
\begin{align}
\tfrac{1}{2}\der[]{t}\|\psi (t)\|^2_{L^2(\Omega)} +\|\bm{D}_{\varrho}^{1/2}\nabla \psi \|_{L^2(\Omega)}^2 \leq \int \mathrm{d}\vec{r}\,|\bm{D}_{\varrho}^{1/2}\nabla \psi | |  \bm{D}_{\varrho}^{1/2}(\varrho_\infty\nabla V_2^{eq}\star \psi +\psi\nabla V_2^{eq}\star \varrho)|. 
\end{align}
Using H{\"o}lder's inequality on the right hand side this becomes
\begin{align}
\tfrac{1}{2}\der[]{t}\|\psi (t)\|^2_{L^2(\Omega)} +\|\bm{D}_{\varrho}^{1/2}\nabla \psi \|_{L^2(\Omega)}^2 \leq \|\bm{D}_{\varrho}^{1/2}\nabla \psi \|_{L^2(\Omega)}\|  \bm{D}_{\varrho}^{1/2}(\varrho_\infty\nabla V_2^{eq}\star \psi +\psi\nabla V_2^{eq}\star \varrho)\|_{L^2(\Omega)}. 
\end{align}
Now using Young's inequality twice on the right hand side we obtain
\begin{align}
&\tfrac{1}{2}\der[]{t}\|\psi (t)\|^2_{L^2(\Omega)} +\|\bm{D}_{\varrho}^{1/2}\nabla \psi \|_{L^2(\Omega)}^2 \leq \tfrac{1}{2}\|\bm{D}_{\varrho}^{1/2}\nabla \psi \|_{L^2(\Omega)}^2 
+  \tfrac{1}{2} \|\bm{D}_{\varrho}^{1/2}(\varrho_\infty\nabla V_2^{eq}\star \psi +\psi\nabla V_2^{eq}\star \varrho)\|_{L^2(\Omega)}^2 \nonumber \\
& \quad\leq \tfrac{1}{2}\|\bm{D}_{\varrho}^{1/2}\nabla \psi \|_{L^2(\Omega)}^2 +  \|\varrho_\infty\bm{D}_{\varrho}^{1/2}\nabla V_2^{eq}\star \psi \|^2_{L^2(\Omega)} +   \|\psi\bm{D}_{\varrho}^{1/2}\nabla V_2^{eq}\star \varrho \|^2_{L^2(\Omega)}.\label{eq:l2_trend_inq_1}
\end{align}
From the positive definiteness and boundedness of the diffusion tensor, we have $\mu_{\min}\leq \|\bm{D}_{\varrho}\|_{L^\infty(\Omega)}\leq \mu_{\max}$.

We also have the following bounds in terms of $\|\psi\|_{L^2(\Omega)}^2$
\begin{align}
&\|\psi\bm{D}_{\varrho}^{1/2}\nabla V_2^{eq}\star \varrho \|^2_{L^2(\Omega)}\leq\mu_{\max}\|\nabla V_2^{eq}\|_{L^\infty(\Omega)}^2\|\psi\|_{L^2(\Omega)}^2,\label{eq:l2_trend_formula_1a}\\
&\|\varrho_\infty\bm{D}_{\varrho}^{1/2}\nabla V_2^{eq}\star \psi \|^2_{L^2(\Omega)}\leq |\Omega|\mu_{\max}\|\varrho_\infty\|_{L^2(\Omega)}^2\|\nabla V_2^{eq}\|^2_{L^\infty(\Omega)}\|\psi \|_{L^2(\Omega)}^2\label{eq:l2_trend_formula_1b}
\end{align}
where $|\Omega|$ denotes the size of $\Omega$ and in \eqref{eq:l2_trend_formula_1a} we have used that
$\nabla V_2^{eq} \star \varrho \leq  \| \nabla V_2^{eq}\|_{L^\infty(\Omega)} \| \varrho \|_{L^1(\Omega)}$ and the fact that $\varrho$ is a probability density with
$\| \varrho \|_{L^1} =1$ (see Corollary \ref{cor:L_1_varrho_is_1}). 
To obtain \eqref{eq:l2_trend_formula_1b} we use that
\[
	\|\varrho_\infty\bm{D}_{\varrho}^{1/2}\nabla V_2^{eq}\star \psi \|^2_{L^2(\Omega)}\leq
	\mu_{\max}\|\varrho_\infty\|_{L^2(\Omega)}^2\|\nabla V_2^{eq}\|^2_{L^\infty(\Omega)}
	\int  \mathrm{d}\vec{r}\, \Big| \rho_\infty(\vec{r}) \int \mathrm{d}\vec{r} ' \, \psi(\vec{r}') \Big|^2.
\]
We then note that, by H\"older's inequality, $ \int \mathrm{d}\vec{r} ' \, \psi(\vec{r}') \leq \|\psi\|_{L^2}\|1\|_{L^2(\Omega)} = 
|\Omega|^{1/2} \|\psi\|_{L^2}$, which gives the result. For \eqref{eq:l2_trend_formula_1b} it remains to bound the non explicit stationary distribution $\varrho_\infty$ in $L^2(\Omega)$, to do this we observe that by the self-consistency equation \eqref{eq:self_cons_eq}
\begin{align}
\|\varrho_\infty\|_{L^2(\Omega)}^2\leq \frac{|\Omega|\times e^{2\|V_2^{eq}\|_{L^\infty}}}{|\Omega|^2\times e^{-2\|V_2^{eq}\|_{L^\infty}}}.\label{eq:l2_trend_formula_2}
\end{align}

Using \eqref{eq:l2_trend_formula_1a}, \eqref{eq:l2_trend_formula_2} and the bounds on $\bm{D}_{\varrho}$, inequality \eqref{eq:l2_trend_inq_1} becomes
\begin{align}
\tfrac{1}{2}\der[]{t}\|\psi (t)\|^2_{L^2(\Omega)} \leq -\tfrac{\mu_{\min}}{2}\|\nabla \psi \|_{L^2(\Omega)}^2 + \mu_{\max}(e^{4\|V_2^{eq}\|_{L^\infty}}+1)\|\nabla V_2^{eq}\|_{L^{\infty}(\Omega)}^2\|\psi \|^2_{L^2(\Omega)}. 
\end{align}
Now since $\psi$ has mean zero we may use the Poincar{\' e}--Wirtinger inequality to write
\begin{align}
\der[]{t}\|\psi (t)\|^2_{L^2(\Omega)}&\leq -\mu_{\min}c^{-2}_{pw}\|\psi \|_{L^2(\Omega)}^2 + 2\mu_{\max}(e^{4\|V_2^{eq}\|_{L^\infty}}+1)\|\nabla V_2^{eq}\|_{L^{\infty}(\Omega)}^2\|\psi \|^2_{L^2(\Omega)}.
\end{align}
Finally, by Gr{\"o}nwall's lemma \cite{evans2002partial}, we obtain 
\begin{align}
\|\psi (t)\|^2_{L^2(\Omega)} \leq \|\psi(0)\|_{L^2(\Omega)}^2\exp\left\lbrace -(\hat{\mu}_{\min}(t)c^{-2}_{pw}- 2\hat{\mu}_{\max}(t)(e^{4\|V_2^{eq}\|_{L^\infty}}+1)\|\nabla V_2^{eq}\|_{L^{\infty}(\Omega)}^2)\right\rbrace,  \label{eq:l2_trend_gronwall_1}
\end{align}
where $\hat{\mu}_{\min}(t) = \int_0^t\mathrm{d}s\,\mu_{\min}(s)$ and $\hat{\mu}_{\max}(t) = \int_0^t\mathrm{d}s\,\mu_{\max}(s)$. Therefore for a unique stationary density $\varrho_\ast$, the necessary condition for exponential convergence $\varrho\to\varrho_\ast$ in $L^2(\Omega)$ as $t\to\infty$ is
\begin{equation}\label{eq:gen_kappa_2_ineq_for_exp_convergence}
\hat{\mu}_{\min}(t)c^{-2}_{pw}- 2\hat{\mu}_{\max}(t)(e^{4\|V_2^{eq}\|_{L^\infty}}+1)\|\nabla V_2^{eq}\|_{L^{\infty}(\Omega)}^2>0.
\end{equation}

It will now be seen that, under the assumption that $\varrho_\infty$ is the unique stationary density with
$\|V_2^{eq}\|_{L^\infty} \leq 1$, we may obtain an explicit condition for the $\|\psi (t)\|^2_{L^2(\Omega)}$. In particular \eqref{eq:l2_trend_gronwall_1} becomes
\begin{align}
\|\psi (t)\|^2_{L^2(\Omega)}\leq \|\psi(0)\|_{L^2(\Omega)}^2 \exp\left\lbrace -(\hat{\mu}_{\min}(t)c^{-2}_{pw}- 2\hat{\mu}_{\max}(t)(e+1)\|\nabla V_2^{eq}\|_{L^{\infty}(\Omega)}^2)\right\rbrace.  
\end{align}\label{eq:l2_trend_gronwall_2}
Then to ensure the argument in the exponential remains negative, we require
\begin{align}
1<\frac{\hat{\mu}_{\min}(t)c_{pw}^{-2}\|\nabla V_2^{eq}\|^{-2}_{L^\infty}}{2(1+e)\hat{\mu}_{\max}(t)}.
\end{align}
This completes the proof of the theorem.
\end{proof}

\section{Results for Existence \& Uniqueness}\label{app:results_ex_uni}

We begin by determining some useful results: first, that $\varrho(\vec{r},t)$ is bounded above in $L^1(\Omega)$ for all time by initial data $\varrho_0$ and second, the $L^1(\Omega)$ norm of $\varrho$ is unity for all time and $\varrho(\vec{r},t)$ is non-negative. The non-negativity is strengthened to strict positivity of $\varrho(\vec{r},t)$ in Section \ref{subsec:strict_pos_rho}. The results in this section are analogous to those in \cite{chazelle2017well}, \cite{greg_mckean_vlasov_torus} with the difference that the boundary conditions we consider are no-flux the diffusion tensor is non-constant. It will be seen that the natural dual space $H^{-1}(\Omega)$ to $H^1(\Omega)$ is provided by the no-flux condition. This is due to the divergence theorem and the boundary condition $\Xi[\varrho_n]\cdot\vec{n}=0$ on $\partial \Omega \times [0,T]$, there is no boundary term, and the normal characterisation of $H^{-1}=(H^1_0)^\ast$ carries over to $H^{1}(\Omega)$.

\subsection{Existence and Uniqueness Results from Classical PDE}

\begin{theorem}\label{thm:classical_PDE_existence}
To start, let $T>0$ and consider a sequence of linear parabolic equations we introduce \eqref{eq:well_posed_un_ibvp}, the frozen version of \eqref{eq:ddft_dyn_non_gradient_flow}, indexed by $n\in \mathbb{N}$,
\begin{align}\label{eq:well_posed_un_ibvp}
\begin{cases}
\quad\partial_t \varrho_n-\nabla \cdot\left(\bm{D}_{\varrho_{n-2}}\nabla \varrho_n\right)=\nabla \cdot\left(\varrho_n\,\bm{D}_{\varrho_{n-2}} ( \nabla V_1 + \bm{A}[\vec{a}]+ \nabla V_2\star \varrho_{n-1})\right) ,\\
\qquad \qquad \qquad \qquad \Xi [\varrho_n]\cdot\vec{n} = 0 \quad \text{ on } \quad \partial \Omega \times [0, T],\\
\qquad\Xi[\varrho_n]:= \bm{D}_{\varrho_{n-2}}\,(\nabla \varrho_n+ 
\varrho_n\,  (  \nabla V_1(\vec{r},t) + \bm{A}[\vec{a}](\vec{r},t) +   \nabla V_2\star \varrho_{n-1})),\\
\qquad \qquad \qquad \qquad \quad  \varrho_n = \varrho_0 \quad \text{ on }\quad \Omega\times \left\lbrace t=0\right\rbrace
\end{cases}
\end{align}
for $n\geq 1$ and $\varrho_0(\vec{r},t) := \varrho_0(\vec{r})$ some initial data, for all time $t>0$. We assume that $\varrho_0\in C^\infty(\Omega)$, $\varrho\geq 0$ and $\int\mathrm{d}\vec{r} \, \varrho(\vec{r}) = 1$. Each equation is parametrised by $n$, a linear parabolic PDE for the unknown $\varrho_n$. With the assumptions \eqref{ass:D_pos_def_weak_diffable} -\eqref{ass:V1_V2_in_W_1_inf}, there exists a unique $\varrho_n\in C^\infty([0,T]; C^\infty(\Omega))$, for each $n\in \mathbb{N}$. 
\end{theorem}

\subsection{Useful Results.}

We identify the expansion of the absolute value function.
\begin{definition}\label{def:def_of_chi_approx_abs}

Let $\epsilon>0$ and define the convex $C^2$ approximation of $|\cdot|$ by
\begin{align}
\chi_\epsilon(\psi) = \begin{cases}
|\psi| \quad \text{ for } \quad \psi>\epsilon,\\
-\tfrac{\psi^4}{8 \epsilon^3 }+\tfrac{3 \psi^2}{4 \epsilon}+ \tfrac{3 \epsilon}{8} \quad \text{ for } \quad \psi\leq \epsilon.
\end{cases}
\end{align}
\end{definition}

We now present our first result concerning the boundedness of the the $L^1$ norm of $\varrho$ in terms of the initial data $\varrho_0$.

\begin{lemma}

If $ \varrho\in C^1([0,\infty);C^2(\Omega))$ is a solution of \eqref{eq:ddft_dyn_non_gradient_flow} with $\varrho_0\in L^1(\Omega)$ then $\|\varrho(t)\|_{L^1(\Omega)}\leq \|\varrho_0\|_{L^1(\Omega)}$ for all time $t\geq 0$.
\end{lemma}

\begin{proof}
We write $\vec{v}_1(\vec{r},t): =  \nabla V_1(\vec{r},t) + \bm{A}(\vec{r},[\vec{a}],t)$ and suppress time dependence to ease notation. Multiplying \eqref{eq:ddft_dyn_non_gradient_flow} by $\chi_\epsilon'(\varrho)$, integrating and using the divergence theorem and chain rule, we have
\begin{align}
\der[]{t}\int \mathrm{d}\vec{r}\,\chi_\epsilon(\varrho)+\|\bm{D}_{\varrho_{n-2}}^{1/2}\nabla \varrho\,\left( \chi_\epsilon''(\varrho)\right) ^{1/2}\|_{L^2(\Omega)}^2
=-\int\mathrm{d}\vec{r}\, \nabla \varrho\, \chi_\epsilon''(\varrho)\cdot[\varrho\,\bm{D}_{\varrho_{n-2}}(\vec{r}) (\vec{v}_1+ (\nabla V_2\star\varrho)(\vec{r}) )].
\end{align}
Now by H{\"o}lder's inequality and then Young's inequality
\begin{align}
&\der[]{t}\int \mathrm{d}\vec{r}\,\chi_\epsilon(\varrho)+\|\bm{D}_{\varrho_{n-2}}^{1/2}\nabla \varrho\,[\chi_\epsilon''(\varrho)]^{1/2}\|_{L^2(\Omega)}^2\leq \|\bm{D}_{\varrho_{n-2}}^{1/2}\nabla \varrho [\chi_\epsilon''(\varrho)]^{1/2}\|_{L^2(\Omega)}
 \| [\chi_\epsilon''(\varrho)]^{1/2}\varrho \,\bm{D}_{\varrho_{n-2}}^{1/2} (\vec{v}_1+ (\nabla V_2\star \varrho) )\|_{L^2(\Omega)}\nonumber\\
&\quad \leq \tfrac{1}{2}\|\bm{D}_{\varrho_{n-2}}^{1/2}\nabla \varrho [\chi_\epsilon''(\varrho)]^{1/2}\|_{L^2(\Omega)}^2
+\tfrac{1}{2}\| [\chi_\epsilon''(\varrho)]^{1/2}\varrho\, \bm{D}_{\varrho_{n-2}}^{1/2} (\vec{v}_1+ (\nabla V_2\star \varrho) )\|_{L^2(\Omega)}^2.
\end{align}
Note there are no boundary terms due to the condition $\Pi[\varrho]\cdot\vec{n}=0$ on $\partial \Omega$. All together this implies the inequality
\begin{align}
&\der[]{t}\int \mathrm{d}\vec{r}\,\chi_\epsilon(\varrho)+\tfrac{1}{2}\|\bm{D}_{\varrho_{n-2}}^{1/2}\nabla \varrho\,\chi_\epsilon''(\varrho)^{1/2}\|_{L^2(\Omega)}^2\leq \tfrac{1}{2} \| [\chi_\epsilon''(\varrho)]^{1/2}\varrho\, \bm{D}_{\varrho_{n-2}}^{1/2} (\vec{v}_1+ (\nabla V_2\star\varrho))\|_{L^2(\Omega)}^2\nonumber\\
&\quad\leq \tfrac{1}{2}\| \bm{D}_{\varrho_{n-2}}^{1/2} (\vec{v}_1+ (\nabla V_2\star\varrho) )\|_{L^\infty}^2
\| [\chi_\epsilon''(\varrho)]^{1/2} \varrho \|_{L^2}^2 \leq c_0 \| [\chi_\epsilon''(\varrho)]^{1/2} \varrho \|_{L^2}^2 (1+\|\varrho\|_{L^1(\Omega)}^2)\label{eq:ddt_chi_rho_bound}
\end{align}
for the constant $c_0 = 2\|\mu_{\max}^{n-2}\|_{L^\infty([0,T])} \max\{\|\vec{v}_1\|_{L^\infty}^2, \|V_2\|_{L^\infty}^2\}$. It is an elementary calculation to show that 
\begin{align}
\varrho^2\chi_\epsilon''(\varrho) = \tfrac{3\varrho^2}{2\epsilon} - \tfrac{3\varrho^4}{2\epsilon^3}
\end{align}
for $\varrho\leq \epsilon$. With this, and the fact that $\chi''(\varrho) = 0$ for $\varrho>\epsilon$, we have
\begin{align}
\| [\chi_\epsilon''(\varrho)]^{1/2} \varrho \|_{L^2}^2 = \int   \mathrm{d}\vec{r}\, \varrho^2 \chi_\epsilon''(\varrho) \mathbb{I}_{\varrho\leq \epsilon}
+ \int   \mathrm{d}\vec{r}\, \varrho^2 \chi_\epsilon''(\varrho) \mathbb{I}_{\varrho > \epsilon}  = \int   \mathrm{d}\vec{r}\, \frac{3 \varrho^2(\epsilon^2 - \varrho^2)}{2 \epsilon^2} \mathbb{I}_{\varrho\leq \epsilon}
\leq \int   \mathrm{d}\vec{r}\, \frac{3 \epsilon}{2}\mathbb{I}_{\varrho\leq \epsilon} \leq c_1\epsilon \label{eq:sqrt_chi''_rho_L2}
\end{align}
for some constant $c_1$ dependent on $\Omega$. Applying Gr{\"o}nwall's lemma to $\eta(\cdot)$ a non-negative, absolutely continuous function on $[0,T]$ which satisfies for a.e. $t$
\begin{align}
\eta'(t)\leq \phi(t)\eta(t) + \psi(t)
\end{align}
where $\phi$, $\psi$ non-negative and integrable functions on $[0,T]$ gives
\begin{align}\label{eq:statement_gronwall}
\eta(t)\leq e^{\int\mathrm{d}s_0^t\,\phi(s)}\eta(t)\Big(\eta(0)+ \int_0^t\mathrm{d}s\,
\psi(s)\Big).
\end{align}

Observe that $\|\varrho\|_{L^1(\Omega)}\leq \int \mathrm{d}\vec{r}\,\chi_\epsilon(\varrho)$. Using this with \eqref{eq:ddt_chi_rho_bound}, \eqref{eq:sqrt_chi''_rho_L2} and \eqref{eq:statement_gronwall} with $\eta(t) = \phi(t) =c_1\epsilon \int \mathrm{d}\vec{r}\,\chi_\epsilon(\varrho)$ and $\psi(t) = c_1\epsilon$
we obtain
\begin{align}
\int \mathrm{d}\vec{r}\,\chi_\epsilon(\varrho)\leq \left(\int \mathrm{d}\vec{r}\,\chi_\epsilon(\varrho_0)+c_1\epsilon \,t\right)\, e^{c_1\epsilon\int_0^t\mathrm{d}s\,\int \mathrm{d}\vec{r}\,\chi_\epsilon(\varrho(\vec{r},s))}.
\end{align}
Now since $\varrho$ is assumed to be continuous in time on $[0,\infty)$ the integral in the exponential is finite. Therefore taking $\epsilon \to 0$ one obtains
\begin{align}
\|\varrho\|_{L^1}\leq \|\varrho_0\|_{L^1}
\end{align}
for every $t>0$.
\end{proof}

\begin{corollary}\label{cor:L_1_varrho_is_1}

If $ \varrho\in C^1([0,\infty);C^2(\Omega))$ is a solution of \eqref{eq:ddft_dyn_non_gradient_flow} with $\varrho_0$ a probability density, that is $\varrho_0\geq 0$ and $\int  \mathrm{d}\vec{r}\,\varrho_0(\vec{r})=1$, then $\|\varrho(t)\|_{L^1(\Omega)}=1$ and $\varrho(t)\geq 0$ in $\Omega$ for all time $t\geq 0$.
\end{corollary}

\begin{proof}
The argument is a standard one. Since, due to no-flux boundary conditions, we have that $ \der[]{t} \int \mathrm{d}\vec{r} \, \rho(\vec{r},t) = 0$,
and that
\begin{align}
1  = \int \mathrm{d}\vec{r}\, \varrho_0(\vec{r}) = \int \mathrm{d}\vec{r}\, \varrho(\vec{r},t) \leq \|\varrho(t)\|_{L^1(\Omega)} \leq \|\varrho(0)\|_{L^1(\Omega)} = \int \mathrm{d}\vec{r}\, \varrho_0(\vec{r}) = 1,
\end{align}
so $\|\varrho(t)\|_{L^1(\Omega)} = 1$. Also observe the two equalities
\begin{align}
&1 = \int \mathrm{d}\vec{r}\, \varrho(\vec{r},t) = \int \mathrm{d}\vec{r}\, \varrho(\vec{r},t)\mathbb{I}_{\varrho\geq 0} + \int \mathrm{d}\vec{r}\, \varrho(\vec{r},t)\mathbb{I}_{\varrho< 0},\nonumber\\
&1 = \int \mathrm{d}\vec{r}\, |\varrho(\vec{r},t)| = \int \mathrm{d}\vec{r}\, \varrho(\vec{r},t)\mathbb{I}_{\varrho\geq 0} - \int \mathrm{d}\vec{r}\, \varrho(\vec{r},t)\mathbb{I}_{\varrho< 0},\nonumber
\end{align}
where in the second line we have used the definition of the absolute value function. Subtracting these equalities we obtain
\begin{align}
2\int \mathrm{d}\vec{r}\, \varrho(\vec{r},t)\mathbb{I}_{\varrho< 0} = 0
\end{align}
which implies $ \varrho(\vec{r},t)\geq 0$ almost everywhere in $\Omega$.   Non-negativity of $\varrho$ on all of $\Omega$ follows from continuity.
\end{proof}

\subsection{Energy Estimates.}

We now obtain uniform estimates on $\varrho_n$ in terms of the initial data $\varrho_0$ in all the required energy norms. The detailed calculations follow \cite{chazelle2017well} but take into account the confining potential and non-constant diffusion tensor $\bm{D}_{\varrho_{n-2}}$. The explicit calculations can be found in RDMW's PhD thesis \cite{rdmwthesisddft}. The first estimate is in $L^{\infty}([0,T];L^2(\Omega))$ and $L^{2}([0,T];H^1(\Omega))$ norms.
\\
\begin{proposition}\label{prop:bound_rhon}
\par
Let $T>0$ and suppose $\{\varrho_n\}_{n\geq 1}$ satisfies \eqref{eq:well_posed_un_ibvp} with $\varrho_0\in C^{\infty}(\Omega)$ a probability density. Then there exists a constant $C(T)$, dependent on time and $\|\mu_{\max}^{n-2}\|_{L^\infty([0,T])}$, such that
\begin{align}\label{eq:L2_H_1_bound_for_rhon}
\|\varrho_n\|_{L^{\infty}([0,T];L^2(\Omega))}+\|\varrho_n\|_{L^{2}([0,T];H^1(\Omega))}\leq C(T,\|\mu_{\max}^{n-2}\|_{L^\infty([0,T])}) \|\varrho_{0}\|_{L^2(\Omega)}.
\end{align}
\end{proposition}

The second estimate is for $L^{\infty}([0,T]; H^1(\Omega))$ and $L^2([0,T];L^2(\Omega))$ norms.
\\
\begin{proposition} \label{prop:un_H1_rho0_bound}
\par
Let $T>0$ and suppose $\{\varrho_n\}_{n\geq 1}$ satisfies \eqref{eq:well_posed_un_ibvp} with  $\varrho_0\in C^{\infty}(\Omega)$ a probability density. Then there exists some constant dependent on time $C(T)$ such that
\begin{align}
&\|\varrho_n\|_{L^{\infty}([0,T]; H^1(\Omega))}+\|\nabla \cdot[\bm{D}_{\varrho_{n-2}}\,\nabla \varrho_n]\|_{L^2([0,T];L^2(\Omega))}^2 \\
&\quad \leq C(T)(\|\varrho_0\|_{H^1(\Omega)}^2+(1+\|\varrho_0\|_{L^2(\Omega)})\|\varrho_0\|_{L^2(\Omega)})^{1/2}.
\end{align}
\end{proposition}

We now have strong convergence of $\left(\varrho_n\right)_{n=1}^{\infty}$, by showing it is a Cauchy sequence in a complete metric space.
\\
\begin{lemma}[$\{\varrho_n\}_{n=1}^{\infty}$ is a Cauchy sequence]\label{lem:rhon_is_cauchy}
\par
Let $T>0$ and suppose $\{\varrho_n\}_{n\geq1}$ satisfies \eqref{eq:well_posed_un_ibvp} with $\varrho_0\in C^\infty(\Omega)$. Then there exists $\varrho\in L^1([0,T];L^1(\Omega))$ such that $\varrho_n\to \varrho$ in $L^1([0,T];L^1(\Omega))$.
\end{lemma}

Lastly we have the uniform estimate on the limit point $\varrho(\vec{r},t)$ in terms of the initial data $\varrho_0$.
\\
\begin{lemma}\label{lem:weak_conv_results}
\par
One has $\varrho\in L^2([0,T]; H^1(\Omega))\cap L^\infty([0,T]; L^2(\Omega))$ and that $\partial_t\varrho\in L^2([0,T]; H^{-1}(\Omega))$ with the uniform bound
\begin{align}\label{eq:total_energy_bound}
\|\varrho\|_{L^{\infty}([0,T];L^2(\Omega))}+\|\varrho\|_{L^{2}([0,T];H^1(\Omega))}+\|\partial_t\varrho\|_{L^2([0,T]; H^{-1}(\Omega))}\leq C(T) \|\varrho_{0}\|_{L^2(\Omega)}.
\end{align}
Additionally there exists a subsequence $\{\varrho_{n_k}\}_{k\geq 1}$ such that 
\begin{align}
\varrho_{n_k}\rightharpoonup \varrho& \quad \text{ in } L^2([0,T]; H^1(\Omega)),\label{lim:rho_nk_in_H1_c}\\
\partial_t\varrho_{n_k}\rightharpoonup \partial_t\varrho& \quad \text{ in } L^2([0,T]; H^{-1}(\Omega)).\label{lim:partial_trho_nk_in_H_minus_1}
\end{align}
where $\rightharpoonup$ denotes weak convergence.
\end{lemma}

The nature of convergence of the sequence $\{\varrho_n\}_{n\geq 1}$ as $n\to\infty$ are consolidated into the following result.
\\
\begin{corollary}\label{cor:summary_of_convergence}
\par
There exists a subsequence $\{\varrho_{n_k}\}_{k\geq 1}\subset \{\varrho_n\}_{n\geq 1}$ and a function $\varrho\in L^2([0,T];H^1(\Omega))$ with $\partial_t\varrho\in L^2([0,T];H^{-1}(\Omega))$ such that
\begin{align}
\varrho_n\to \varrho& \quad \text{ in } L^1([0,T]; L^1(\Omega)),\\
\varrho_{n_k}\rightharpoonup \varrho& \quad \text{(weakly) in } L^2([0,T]; H^1(\Omega)),\\
\partial_t \varrho_{n_k}\rightharpoonup \partial_t\varrho& \quad \text{(weakly) in } L^2([0,T]; H^{-1}(\Omega)).
\end{align}
\end{corollary}

We are now in the position to obtain the existence and uniqueness of weak solutions $\varrho(\vec{r},t)$. First we state a calculus result which will be useful when working with the weak formulation \eqref{eq:weak_formulation_pair_rho}.
\\
\begin{lemma}\label{lem:calculus_of_inner_product}
\par
Suppose $\varrho\in L^2([0,T]; H^1(\Omega))$ and $\partial_t\varrho\in L^2([0,T]; H^{-1}(\Omega))$ then the mapping
\begin{align}
t\to \|\varrho(t)\|_{L^2(\Omega)}^2
\end{align}
is absolutely continuous with
\begin{align}
\der[]{t}\|\varrho(t)\|_{L^2(\Omega)}^2 = 2\langle \partial_t\varrho(t),\, \varrho(t)\rangle
\end{align}
for a.e. $t\in [0,T]$.
\end{lemma}

\begin{proof}
Since the condition no-flux boundary condition \eqref{eq:ddft_dyn_non_gradient_flow_bc} guarantees integration by parts without extra terms the proof is identical to the textbook one \cite{evans2002partial}.
\end{proof}

\begin{theorem}[Lax--Milgram Theorem]\label{thm:lax-milgram}
Let $X\subset \mathbb{R}^d$ be a Hilbert space and $a:X\times X\to\mathbb{R}$ be a continuous bilinear functional. Suppose 
\begin{enumerate}
\item $a$ is coercive such that $| a(u,u) | \geq c\|u\|_{X}$ for some $c_1>0$ and every $u\in X$.
\item $a$ is bounded, that is for every $u,v\in X$ we have $|a(u,v)| \leq c_2\|u\|_{X}\|v\|_{X}$ for some $c_2>0$.
\item $l(v)$ is bounded, that is for every $v\in X$ there exists some $c_3>0$ such that $| l(v)|\leq c_3\|v\|_{X}$.
\end{enumerate}
Then, there exists a unique $u\in X$ such that
\begin{align}
B(u,v) = l(v),
\end{align}
for every $v\in X$.
\end{theorem}

\bibliographystyle{plain}


\begin{thebibliography}{10}

\bibitem{al2018dynamical}
H.~M. Al-Saedi, A.~J. Archer, and J.~Ward.
\newblock Dynamical density-functional-theory-based modeling of tissue
  dynamics: Application to tumor growth.
\newblock {\em Phys. Rev. E}, 98(2):022407, 2018.

\bibitem{Archer05}
A.~J. Archer.
\newblock Dynamical density functional theory: binary phase-separating
  colloidal fluid in a cavity.
\newblock {\em J. Phys.: Condens. Matter}, 17(10):1405, 2005.

\bibitem{archer2009dynamical}
A.~J. Archer.
\newblock Dynamical density functional theory for molecular and colloidal
  fluids: A microscopic approach to fluid mechanics.
\newblock {\em J. Chem. Phys.}, 130(1):014509, 2009.

\bibitem{ArcherEvans04}
A.~J. Archer and R.~Evans.
\newblock Dynamical density functional theory and its application to spinodal
  decomposition.
\newblock {\em J. Chem. Phys.}, 121(9):4246--4254, 2004.

\bibitem{ArcherRobbinsThiele10}
A.~J. Archer, M.~J. Robbins, and U.~Thiele.
\newblock Dynamical density functional theory for the dewetting of evaporating
  thin films of nanoparticle suspensions exhibiting pattern formation.
\newblock {\em Phys. Rev. E}, 81:021602, Feb 2010.

\bibitem{berim2009simple}
G.~O. Berim and E.~Ruckenstein.
\newblock Simple expression for the dependence of the nanodrop contact angle on
  liquid-solid interactions and temperature.
\newblock {\em J. Chem. Phys.}, 130(4):044709, 2009.

\bibitem{bogachev2015fokker}
V.~I. Bogachev, N.~V. Krylov, M.~R{\"o}ckner, and S.~V. Shaposhnikov.
\newblock {\em Fokker-Planck-Kolmogorov Equations}, volume 207.
\newblock American Mathematical Soc, 2015.

\bibitem{camazine2003self}
S.~Camazine, J.~Deneubourg, N.~R. Franks, J.~Sneyd, E.~Bonabeau, and
  G.~Theraula.
\newblock {\em Self-organization in biological systems}, volume~7.
\newblock Princeton university press, 2003.

\bibitem{canizo2010collective}
J.~A. Canizo, J.~A. Carrillo, and J.~Rosado.
\newblock Collective behavior of animals: Swarming and complex patterns.
\newblock {\em Arbor}, 186:1035--1049, 2010.

\bibitem{carrillo2009double}
J.~A. Carrillo, M.~R. D’orsogna, and V.~Panferov.
\newblock Double milling in self-propelled swarms from kinetic theory.
\newblock {\em Kinet. Relat. Mod.}, 2(2):363--378, 2009.

\bibitem{greg_mckean_vlasov_torus}
J.~A. Carrillo, R.~S. Gvalani, G.~A. Pavliotis, and A.~Schlichting.
\newblock Long-time behaviour and phase transitions for the
  {M}c{K}ean--{V}lasov equation on the torus.
\newblock {\em Arch. Ration. Mech. An.}, Jul 2019.

\bibitem{chayes2010mckean}
L.~Chayes and V.~Panferov.
\newblock The {M}c{K}ean--{V}lasov equation in finite volume.
\newblock {\em J. Stat. Phys.}, 138(1-3):351--380, 2010.

\bibitem{chazelle2017well}
B.~Chazelle, Q.~Jiu, Q.~Li, and C.~Wang.
\newblock Well-posedness of the limiting equation of a noisy consensus model in
  opinion dynamics.
\newblock {\em J. Differ. Equations.}, 263(1):365--397, 2017.

\bibitem{degond1986global}
P.~Degond.
\newblock Global existence of smooth solutions for the
  {V}lasov-{F}okker-{P}lanck equation in $1 $ and $2 $ space dimensions.
\newblock In {\em Annales scientifiques de l'{\'E}cole Normale Sup{\'e}rieure},
  volume~19, pages 519--542. Elsevier, 1986.

\bibitem{dressler1987stationary}
K.~Dressler.
\newblock Stationary solutions of the {V}lasov-{F}okker-{P}lanck equation.
\newblock {\em Math. Method. Appl. Sci.}, 9(1):169--176, 1987.

\bibitem{Einstein06}
A.~Einstein.
\newblock Eine neue bestimmung der molek{\"u}ldimensionen.
\newblock {\em Ann. Phys. Lpz.}, 324(2):289--306, 1906.

\bibitem{evans2002partial}
L.~C. Evans.
\newblock Partial differential equations, ams, providence, ri, 1998.
\newblock {\em Grad. Stud. Math.}, 19, 2002.

\bibitem{fredholm1900nouvelle}
I.~Fredholm.
\newblock Sur une nouvelle m{\'e}thode pour la r{\'e}solution du probl{\`e}me
  de dirichlet.
\newblock {\em {\"O}fver. Vet. Akad. F{\"o}rhand}, 57:39--46, 1900.

\bibitem{GMPEQDDFT21}
B.~D. Goddard, R.~D. Mills-Williams, and G.~A. Pavliotis.
\newblock Phase transitions for dynamic density functional theory.
\newblock {\em In preparation}, 2021.

\bibitem{goddard2012unification}
B.~D. Goddard, A.~Nold, N.~Savva, P.~Yatsyshin, and S.~Kalliadasis.
\newblock Unification of dynamic density functional theory for colloidal fluids
  to include inertia and hydrodynamic interactions: derivation and numerical
  experiments.
\newblock {\em J. Phys.: Condens. Matter}, 25(3):035101, 2012.

\bibitem{goddard2012overdamped}
B.~D. Goddard, G.~A. Pavliotis, and S.~Kalliadasis.
\newblock The overdamped limit of dynamic density functional theory: Rigorous
  results.
\newblock {\em Multiscale. Model. Sim.}, 10(2):633--663, 2012.

\bibitem{gorban2014hilbert}
A.~Gorban and I.~Karlin.
\newblock Hilbert{'}s 6th problem: exact and approximate hydrodynamic manifolds
  for kinetic equations.
\newblock {\em B. Am. Math. Soc.}, 51(2):187--246, 2014.

\bibitem{happel2012low}
J.~Happel and H.~Brenner.
\newblock {\em Low {R}eynolds number hydrodynamics: with special applications
  to particulate media}, volume~1.
\newblock Springer Science \& Business Media, 2012.

\bibitem{HodaKumar07}
N.~Hoda and S.~Kumar.
\newblock Brownian dynamics simulations of polyelectrolyte adsorption in shear
  flow with hydrodynamic interaction.
\newblock {\em J. Chem. Phys}, 127(23):234902, 2007.

\bibitem{HuberKoehlerYang11}
G.~Huber, S.~A. Koehler, and J.~Yang.
\newblock Micro-swimmers with hydrodynamic interactions.
\newblock {\em Math. Comput. Model.}, 53(7):1518--1526, 2011.

\bibitem{kim2013microhydrodynamics}
S.~Kim and S.~J. Karrila.
\newblock {\em Microhydrodynamics: principles and selected applications}.
\newblock Courier Corporation, 2013.

\bibitem{kuznetsov2015sharp}
N.~Kuznetsov and A.~Nazarov.
\newblock Sharp constants in the {P}oincar{\'e}, {S}teklov and related
  inequalities (a survey).
\newblock {\em Mathematika}, 61(2):328--344, 2015.

\bibitem{LaugaSquires05}
E.~Lauga and T.~M. Squires.
\newblock Brownian motion near a partial-slip boundary: A local probe of the
  no-slip condition.
\newblock {\em Phys. Fluids}, 17(10):103102, 2005.

\bibitem{lax2014functional}
P.~D. Lax.
\newblock {\em Functional Analysis}.
\newblock Pure and Applied Mathematics: A Wiley Series of Texts, Monographs and
  Tracts. Wiley, 2014.

\bibitem{lax1954ix}
P.D. Lax and A.N. Milgram.
\newblock Ix. parabolic equations.
\newblock {\em Contributions to the theory of partial differential equations},
  (33):167, 1954.

\bibitem{Lutsko10}
J.~F. Lutsko.
\newblock Recent developments in classical density functional theory.
\newblock {\em Adv. Chem. Phys.}, 144:1, 2010.

\bibitem{Lutsko12}
J.~F. Lutsko.
\newblock A dynamical theory of nucleation for colloids and macromolecules.
\newblock {\em J. Chem. Phys.}, 136(3):034509, 2012.

\bibitem{MarconiMelchionna07}
U.~M.~B. Marconi and S.~Melchionna.
\newblock Phase-space approach to dynamical density functional theory.
\newblock {\em {J. Chem. Phys.}}, {126}:{184109}, {2007}.

\bibitem{MarconiTarazonaCecconi07}
U.~M.~B. Marconi, P.~Tarazona, and F.~Cecconi.
\newblock Theory of thermostatted inhomogeneous granular fluids: A
  self-consistent density functional description.
\newblock {\em {J. Chem. Phys.}}, {126}:{164904}, {2007}.

\bibitem{Mermin:1965lo}
N.~D. Mermin.
\newblock Thermal properties of the inhomogeneous electron gas.
\newblock {\em Phys. Rev.}, 137:A1441, 1965.

\bibitem{rdmwthesisddft}
R.~D. Mills-Williams.
\newblock {\em Analysis and Applications of Dynamic Density Functional Theory}.
\newblock PhD thesis, 2020.

\bibitem{pavliotis2014stochastic}
G.~A. Pavliotis.
\newblock {\em Stochastic processes and applications: diffusion processes, the
  Fokker-Planck and Langevin equations}, volume~60.
\newblock Springer, 2014.

\bibitem{pavliotis2008multiscale}
G.~A. Pavliotis and A.~Stuart.
\newblock {\em Multiscale methods: averaging and homogenization}.
\newblock Springer Science \& Business Media, 2008.

\bibitem{payne1960optimal}
L.~E. Payne and H.~F. Weinberger.
\newblock An optimal {P}oincar{\'e} inequality for convex domains.
\newblock {\em Arch. Ration. Mech. An.}, 5(1):286--292, 1960.

\bibitem{PennaDzubiellaTarazona03}
F.~Penna, J.~Dzubiella, and P.~Tarazona.
\newblock Dynamic density functional study of a driven colloidal particle in
  polymer solutions.
\newblock {\em Phys. Rev. E}, 68(6):061407, 2003.

\bibitem{pereira2012equilibrium}
A.~Pereira and S.~Kalliadasis.
\newblock Equilibrium gas--liquid--solid contact angle from density-functional
  theory.
\newblock {\em J. Fluid. Mech.}, 692:53--77, 2012.

\bibitem{rex2009dynamical}
M.~Rex and H.~L{\"o}wen.
\newblock Dynamical density functional theory for colloidal dispersions
  including hydrodynamic interactions.
\newblock {\em Eur. Phys. J. E. Soft Matter and Biological Physics.},
  28(2):139--146, 2009.

\bibitem{rotne1969variational}
J.~Rotne and S.~Prager.
\newblock Variational treatment of hydrodynamic interaction in polymers.
\newblock {\em J. Chem. Phys.}, 50(11):4831--4837, 1969.

\bibitem{schmidt2018power}
M.~Schmidt.
\newblock Power functional theory for {N}ewtonian many-body dynamics.
\newblock {\em J. Chem. Phys.}, 148(4):044502, 2018.

\bibitem{schmidt2013power}
M.~Schmidt and J.~M. Brader.
\newblock Power functional theory for {B}rownian dynamics.
\newblock {\em J. Chem. Phys.}, 138(21):214101, 2013.

\bibitem{sibley2013contact}
D.~N. Sibley, A.~Nold, N.~Savva, and S.~Kalliadasis.
\newblock The contact line behaviour of solid-liquid-gas diffuse-interface
  models.
\newblock {\em Phys. Fluids}, 25(9):092111, 2013.

\bibitem{vanTeeffelenLikosLowen08}
S.~van Teeffelen, C.~N. Likos, and H.~L\"owen.
\newblock Colloidal crystal growth at externally imposed nucleation clusters.
\newblock {\em Phys. Rev. Lett.}, 100:108302, Mar 2008.

\bibitem{WittkowskiLowenBrand10}
R.~Wittkowski, H.~L\"owen, and H.~R. Brand.
\newblock Derivation of a three-dimensional phase-field-crystal model for
  liquid crystals from density functional theory.
\newblock {\em Phys. Rev. E}, 82:031708, Sep 2010.

\bibitem{wouk1966note}
A.~Wouk.
\newblock A note on square roots of positive operators.
\newblock {\em SIAM. Rev.}, 8(1):100--102, 1966.

\bibitem{WysockiLowen11}
A.~Wysocki and H.~L{\"o}wen.
\newblock Effects of hydrodynamic interactions in binary colloidal mixtures
  driven oppositely by oscillatory external fields.
\newblock {\em J. Phys.: Condens. Matter}, 23(28):284117, 2011.

\end{thebibliography}

          \end{document}